\documentclass[11pt,draftcls,onecolumn]{IEEEtran}
\usepackage{amsmath, amssymb, epsfig, graphicx, bm}
\usepackage{float, subfigure}
\usepackage{color}
\usepackage{amsthm}
\usepackage{amsfonts}
\usepackage{cite, url}
\usepackage{color,hyperref}
\definecolor{darkblue}{rgb}{0.0,0.0,0.4}
\hypersetup{colorlinks,breaklinks,
            linkcolor=darkblue,urlcolor=darkblue,
            anchorcolor=darkblue,citecolor=darkblue}

\newtheorem{theorem}{Theorem}

\newtheorem{assumption}{Assumption}

\begin{document}

\title{\huge{On the Linear Convergence of the ADMM in Decentralized Consensus Optimization}}
\author{Wei Shi, Qing Ling, Kun Yuan, Gang Wu, and Wotao Yin
\thanks{Copyright (c) 2013 IEEE. Personal use of this material is
permitted. However, permission to use this material for any other
purposes must be obtained from the IEEE by sending a request to
pubs-permissions@ieee.org.

W. Shi, Q. Ling, K. Yuan, and G. Wu are with the Department of
Automation, University of Science and Technology of China, Hefei,
Anhui, China, 230026. W. Yin is with the Department of
Mathematics, University of California, Los Angeles, California,
USA, 90095. Corresponding author: Qing Ling. Email:
\href{mailto:qingling@mail.ustc.edu.cn}{qingling@mail.ustc.edu.cn}. Wei Shi is supported by Chinese
Scholarship Council grant 201306340046. Qing Ling is supported by
NSFC grant 61004137 and Chinese Scholarship Council grant
2011634506. G. Wu is supported by MOF/MIIT/MOST grant
BB2100100015. W. Yin is supported by ARL and ARO grant
W911NF-09-1-0383 and NSF grants DMS-0748839 and DMS-1317602. Part
of this paper appeared in the 38th International Conference on
Acoustics, Speech, and Signal Processing, Vancouver, Canada, May
26--31, 2013 \cite{Shi2013}.}}

\maketitle

\begin{abstract}
In decentralized consensus optimization, a connected network of
agents collaboratively minimize the sum of their local objective
functions over a common decision variable, where their information
exchange is restricted between the neighbors. To this end, one can
first obtain a problem reformulation and then apply the
alternating direction method of multipliers (ADMM). The method
applies iterative computation at the individual agents and
information exchange between the neighbors. This approach has been
observed to converge quickly and deemed powerful. This paper
establishes its linear convergence rate for the decentralized
consensus optimization problem with strongly convex local
objective functions. The theoretical convergence rate is
explicitly given in terms of the network topology, the properties
of local objective functions, and the algorithm parameter. This
result is not only a performance guarantee but also a guideline
toward accelerating the ADMM convergence.
\end{abstract}

\begin{IEEEkeywords}
Decentralized consensus optimization, alternating direction method
of multipliers (ADMM), linear convergence
\end{IEEEkeywords}
\IEEEpeerreviewmaketitle

\section{Introduction}

\IEEEPARstart{R}ecent advances in signal processing and control of
networked multi-agent systems have led to much research interests
in decentralized optimization
\cite{Inalhany2002,Ren2007,Johansson2008,Xiao2007,Dimakis2010,Predd2009,Mateos2010,Schizas2008,Ling2010,Bazerque2010,Bazerque2011,
Kekatos2013,Gan2013}. Decentralized optimization problems arising
in networked multi-agent systems include coordination of aircraft
or vehicle networks \cite{Inalhany2002,Ren2007,Johansson2008},
data processing of wireless sensor networks
\cite{Xiao2007,Dimakis2010,Predd2009,Mateos2010,Schizas2008,Ling2010},
spectrum sensing of cognitive radio networks
\cite{Bazerque2010,Bazerque2011}, state estimation and operation
optimization of smart grids \cite{Kekatos2013,Gan2013}, etc. In
these scenarios, the data is collected and/or stored in a
distributed manner; a fusion center is either disallowed or not
economical. Consequently, any computing tasks must be accomplished
in a decentralized and collaborative manner by the agents. This
approach can be powerful and efficient, as the computing tasks are
distributed over all the agents and information exchange occurs
only between the agents with direct communication links. There is no risk of  central computation overload or network congestion.

In this paper, we focus on \textit{decentralized consensus
optimization}, an important class of decentralized optimization in
which a network of $L$ agents cooperatively solve
\begin{equation} \label{eq:ps}
\begin{array}{cl}
\min\limits_{\tilde{x}} \quad \sum\limits_{i=1}^L f_i(\tilde{x}),
\end{array}
\end{equation}
over a common optimization variable $\tilde{x}$, where
$f_i(\tilde{x}): \mathbb{R}^N \rightarrow \mathbb{R}$ is the
local objective function known by agent $i$. This formulation
arises in averaging \cite{Johansson2008,Xiao2007,Dimakis2010},
learning \cite{Predd2009,Mateos2010}, and estimation
\cite{Schizas2008, Ling2010, Bazerque2010, Bazerque2011,
Kekatos2013} problems. Examples of $f_i(\tilde{x})$ include least
squares \cite{Johansson2008,Xiao2007,Dimakis2010}, regularized
least squares \cite{Mateos2010,Ling2010, Bazerque2010,
Bazerque2011}, as well as more general ones \cite{Predd2009}. The
values of $\tilde{x}$ can stand for average temperature of a room
\cite{Xiao2007,Dimakis2010}, frequency-domain occupancy of spectra
\cite{Bazerque2010,Bazerque2011}, states of a smart grid system
\cite{Kekatos2013,Gan2013}, and so on.

There exist several methods for decentralized consensus
optimization, including distributed subgradient descent algorithms
\cite{Nedic2009,Ram2010,Tsianos2013}, dual averaging methods
\cite{Duchi2012,Tsianos2012}, and the alternating direction method
of multipliers (ADMM)
\cite{Mateos2010,Schizas2008,Ling2010,Erseghe2011,Bertsekas1997}.
Among these algorithms, the ADMM demonstrates fast convergence in
many applications, e.g., \cite{Mateos2010,Schizas2008,Ling2010}.
However, how fast it converges and what factors affect the rate
are both unknown. This paper addresses these issues.

\subsection{Our Contributions}

Firstly, we establish the linear convergence rate of the ADMM that
is applied to decentralized consensus optimization with strongly
convex local objective functions. This theoretical result gives a
performance guarantee for the ADMM and validates the observation
in prior literature.

Secondly, we study how the network topology, the properties of
local objective functions, and the algorithm parameter affect the
convergence rate. The analysis provide guidelines for networking
strategies, objective-function splitting strategies, and algorithm
parameter settings to achieve faster convergence.

\subsection{Related Work}

Besides the ADMM, existing decentralized approaches for solving
(\ref{eq:ps}) include belief propagation \cite{Predd2009},
incremental optimization \cite{Rabbat2006}, subgradient descent
\cite{Nedic2009,Ram2010,Tsianos2013}, dual averaging
\cite{Duchi2012,Tsianos2012}, etc. Belief propagation and
incremental optimization require one to predefine a tree or loop
structure in the network, whereas the advantage of the ADMM,
subgradient descent, and dual averaging is that they do not rely
on any predefined structures. Subgradient descent and dual
averaging work well for asynchronous networks but suffer from slow
convergence. Indeed, for subgradient descent algorithms
\cite{Nedic2009} and \cite{Ram2010} establish the convergence rate
of $O(1/k)$, where $k$ is the number of iterations, to a
neighborhood of the optimal solution when the local subgradients
are bounded and the stepsize is fixed. Further assuming that the
local objective functions are strongly convex, choosing a dynamic
stepsize leads to a rate of $O(\log(k)/k)$ \cite{Tsianos2013}.
Dual averaging methods using dynamic stepsizes also have sublinear
rates, e.g., $O(\log(k)/\sqrt{k})$ as proved in \cite{Duchi2012}
and \cite{Tsianos2012}.

The decentralized ADMM approaches use synchronous steps by all the
agents but have much faster empirical convergence, as demonstrated
in many applications \cite{Mateos2010,Schizas2008,Ling2010}.
However, existing convergence rate analysis of the ADMM is
restricted to the classic, centralized computation. The
centralized ADMM has a sublinear convergence rate $O(1/k)$ for
general convex optimization problems \cite{Yuan2012}. In
\cite{Hong2013} an ADMM with restricted stepsizes is proposed and
proved to be linearly convergent for certain types of non-strongly
convex objective functions. A recent paper \cite{Deng2013} shows a
linear convergence rate $O(1/a^k)$ for some $a > 1$ under a
strongly convex assumption, and our paper extends the analysis
tools therein to the decentralized regime.

A notable work about convergence rate analysis is
\cite{Erseghe2011}, which proves the linear convergence rate of
the ADMM applied to the average consensus problem, a special case
of (\ref{eq:ps}) in which $f_i(\tilde{x})=\|\tilde{x}-y_i\|_2^2$
with $y_i$ being a local measurement vector of agent $i$. Its
analysis takes a state-transition equation approach, which is not
applicable to the more general local objective functions
considered in this paper.

\subsection{Paper Organization and Notation}

This paper is organized as follows. Section \ref{sec: ADMM_opt}
reformulates the decentralized consensus optimization problem and
develops an algorithm based on the ADMM. Section \ref{sec:
rate_analysis} analyzes the linear convergence rate of the ADMM
and shows how to accelerate the convergence through tuning the
algorithm parameter. Section \ref{sec: num_exp} provides extensive
numerical experiments to validate the theoretical analysis in
Section \ref{sec: rate_analysis}. Section \ref{sec: conclusion}
concludes the paper.

In this paper we denote $\|x\|_2$ as the Euclidean norm of a
vector $x$ and $\langle x, y \rangle$ as the inner product of two
vectors $x$ and $y$. Given a semidefinite matrix $G$ with proper
dimensions, the $G$-norm of $x$ is $\sqrt{x^T G x}$. We let
$\sigma_{\max}(G)$ be the operator that returns the largest
singular value of $G$ and $\tilde{\sigma}_{\min}(G)$ be the one
that returns the smallest nonzero singular value of $G$.

We use two kinds of definitions of convergence, Q-linear
convergence and R-linear convergence. We say that a sequence
$y^k$, where the superscript $k$ stands for time index, Q-linearly
converges to a point $y^*$ if there exists a number $\rho \in
(0,1)$ such that
$\lim\limits_{k\rightarrow\infty}\frac{\|y^{k+1}-y^*\|}{\|y^{k}-y^*\|}=\rho$
with $\|\cdot\|$ being a vector norm. We say that a sequence $x^k$
R-linearly converges to a point $x^*$ if for all $k$,
$\|x^{k}-x^*\|\leq\|y^k-y^*\|$ where $y^k$ Q-linearly converges to
$y^*$.

\section{The ADMM for Decentralized Consensus Optimization}\label{sec: ADMM_opt}

In this section, we first reformulate the decentralized consensus
optimization problem (\ref{eq:ps}) such that it can be solved by
the ADMM (see Section \ref{sec:2a}). Then we develop the
decentralized ADMM approach and provide a simplified decentralized
algorithm (see Section \ref{sec:2b}).

\subsection{Problem Formulation}\label{sec:2a}

Throughout the paper, we consider a network consisting of $L$
agents bidirectionally connected by $E$ edges (and thus
$2E$ arcs). We can describe the network as a symmetric directed
graph $\mathcal{G}_\mathrm{d}=\{\mathcal{V}, \mathcal{A}\}$ or an undirected graph $\mathcal{G}_\mathrm{u}=\{\mathcal{V}, \mathcal{E}\}$,
where $\mathcal{V}$ is the set of vertexes with cardinality
$|\mathcal{V}|=L$, $\mathcal{A}$ is the set of arcs with
$|\mathcal{A}|=2E$, and $\mathcal{E}$ is the set of edges with
$|\mathcal{E}|=E$. Algorithms that solve the decentralized
consensus optimization problem (\ref{eq:ps}) are developed based on this graph.

Generally speaking, the ADMM applies to the convex optimization
problem in the form of
\begin{equation} \label{eq:ADMM}
\begin{array}{cl}
\min\limits_{y_1,y_2} &g_1(y_1)+g_2(y_2), \\
\mbox{s.t.} &C_1y_1+C_2y_2=b, \\
\end{array}
\end{equation}
where $y_1$ and $y_2$ are optimization variables, $g_1$ and $g_2$
are convex functions, and $C_1y_1+C_2y_2=b$ is a linear constraint
of $y_1$ and $y_2$. The ADMM solves a sequence of subproblems
involving $g_1$ and $g_2$ one at a time and iterates to converge
as long as a saddle point exists.

To solve (\ref{eq:ps}) with the ADMM in a decentralized manner, we
reformulate it as
\begin{equation} \label{eq:nc}
\begin{array}{cl}
\min\limits_{\{x_i\},\{z_{ij}\}} &\sum\limits_{i=1}^L f_i(x_i), \\
\mbox{s.t.}\ &x_i = z_{ij},~ x_j = z_{ij}, ~\forall (i,j) \in \mathcal{A}. \\
\end{array}
\end{equation}
Here $x_i$ is the local copy of the common optimization variable
$\tilde{x}$ at agent $i$ and $z_{ij}$ is an auxiliary variable
imposing the consensus constraint on neighboring agents $i$ and
$j$. In the constraints, $\{x_i\}$ are separable when $\{z_{ij}\}$
are fixed, and vice versa. Therefore, (\ref{eq:nc}) lends itself
to decentralized computation in the ADMM framework. Apparently, (\ref{eq:nc}) is equivalent to (\ref{eq:ps}) when the
network is connected.

Defining $x \in \mathbb{R}^{LN}$ as a vector concatenating all
$x_i$, $z \in \mathbb{R}^{2EN}$ as a vector concatenating all
$z_{ij}$, and $f(x)=\sum_{i=1}^L f_i(x_i)$, (\ref{eq:nc}) can be
written in a matrix form as
\begin{equation} \label{eq:nc-mat}
\begin{array}{cl}
\min\limits_{x,z} &f(x)+g(z), \\
\mbox{s.t.} &Ax + Bz = 0, \\
\end{array}
\end{equation}
where $g(z)=0$, which fits the form of (\ref{eq:ADMM}), and is
amenable to the ADMM. Here $A=[A_1; A_2]$; $A_1, A_2 \in
\mathbb{R}^{2EN \times LN}$ are both composed of $2E \times L$
blocks of $N \times N$ matrices. If $(i,j) \in \mathcal{A}$ and
$z_{ij}$ is the $q$th block of $z$, then the $(q,i)$th block of
$A_1$ and the $(q,j)$th block of $A_2$ are $N \times N$ identity
matrices $I_{N}$; otherwise the corresponding blocks are $N \times
N$ zero matrices $0_{N}$. Also, we have $B=[-I_{2EN}; -I_{2EN}]$
with $I_{2EN}$ being a $2EN \times 2EN$ identity matrix.

\subsection{Algorithm Development} \label{sec:2b}

Now we apply the ADMM to solve (\ref{eq:nc-mat}). The augmented
Lagrangian of (\ref{eq:nc-mat}) is
\begin{equation} \label{eq:alf}
L_c(x,z,\lambda) = f(x) + \langle \lambda, Ax+Bz \rangle +
\frac{c}{2}\|Ax+Bz\|_2^2, \nonumber
\end{equation}
where $\lambda \in \mathbb{R}^{4EN}$ is the Lagrange multiplier
and $c$ is a positive algorithm parameter. At iteration $k+1$, the
ADMM firstly minimizes $L_c(x,z^k,\lambda^k)$ to obtain $x^{k+1}$,
secondly minimizes $L_c(x^{k+1},z,\lambda^k)$ to obtain $z^{k+1}$,
and finally updates $\lambda^{k+1}$ from $x^{k+1}$ and $z^{k+1}$.
The updates are
\begin{equation} \label{eq:lambda-update}
\begin{array}{cr}
x\text{-update:} & \nabla f(x^{k+1}) + A^T \lambda^k + cA^T(Ax^{k+1}+Bz^k)=0, \\
z\text{-update:} & B^T \lambda^k + cB^T(Ax^{k+1}+Bz^{k+1})=0, \\
\lambda\text{-update:} & \lambda^{k+1}-\lambda^k-c(Ax^{k+1}+Bz^{k+1})=0, \\
\end{array}
\end{equation}
where $\nabla f(x^{k+1})$ is the gradient of $f(x)$ at point
$x=x^{k+1}$ if $f$ is differentiable, or is a subgradient if $f$ is
non-differentiable.

Next we show that if the initial values of $z$ and $\lambda$ are
properly chosen the ADMM updates in (\ref{eq:lambda-update}) can
be simplified (see also the derivation in \cite{Mateos2010}).
Multiplying the two sides of the $\lambda$-update by $A^T$ and adding
it to the $x$-update, we have $\nabla f(x^{k+1}) + A^T
\lambda^{k+1} + cA^TB(z^k-z^{k+1})=0$. Further, multiplying the
two sides of the $\lambda$-update by $B^T$ and adding it to the $z$-update we have $B^T \lambda^{k+1}=0$. Therefore
(\ref{eq:lambda-update}) can be equivalently expressed as
\begin{equation} \label{eq:lambda-update-new}
\begin{array}{cr}
& \nabla f(x^{k+1}) + A^T \lambda^{k+1} + cA^TB(z^k-z^{k+1})=0, \\
& B^T \lambda^{k+1}=0, \\
& \lambda^{k+1}-\lambda^k-c(Ax^{k+1}+Bz^{k+1})=0. \\
\end{array}
\end{equation}
Letting $\lambda = [\beta; \gamma]$ with $\beta, \gamma \in
\mathbb{R}^{2EN}$ and recalling $B=[-I_{2EN}; -I_{2EN}]$, we know
$\beta^{k+1}=-\gamma^{k+1}$ from the second equation of
(\ref{eq:lambda-update-new}). Therefore, the first equation in
(\ref{eq:lambda-update-new}) reduces to $\nabla f(x^{k+1}) + M_-
\beta^{k+1} - cM_+(z^k-z^{k+1})=0$ where $M_+=A_1^T+A_2^T$ and
$M_-=A_1^T-A_2^T$. The third equation in
(\ref{eq:lambda-update-new}) splits to two equations
$\beta^{k+1}-\beta^k-cA_1x^{k+1}+cz^{k+1}=0$ and
$\gamma^{k+1}-\gamma^k-cA_2x^{k+1}+cz^{k+1}=0$. If we choose the
initial value of $\lambda$ as $\beta^0=-\gamma^0$ such that
$\beta^k = -\gamma^k$ holds for $k=0,1,\cdots$, summing and
subtracting these two equations result in
$\frac{1}{2}M_+^Tx^{k+1}-z^{k+1}=0$ and
$\beta^{k+1}-\beta^k-\frac{c}{2}M_-^Tx^{k+1}=0$, respectively. If we further choose the initial value of $z$ as
$z^0=\frac{1}{2}M_+^Tx^0$, $\frac{1}{2}M_+^Tx^k-z^k=0$ holds for
$k=0,1,\cdots$.

To summarize, with initialization $\beta^0=-\gamma^0$ and
$z^0=\frac{1}{2}M_+^Tx^0$, (\ref{eq:lambda-update-new}) reduces to
\begin{equation} \label{eq:beta-update}
\begin{array}{cr}
& \nabla f(x^{k+1}) + M_- \beta^{k+1} - cM_+(z^k-z^{k+1})=0, \\
& \beta^{k+1}-\beta^k-\frac{c}{2}M_-^Tx^{k+1}=0, \\
& \frac{1}{2}M_+^Tx^k-z^k=0. \\
\end{array}
\end{equation}
In Section \ref{sec: rate_analysis} we will analyze the
convergence rate of the ADMM updates (\ref{eq:beta-update}). The
analysis requires an extra initialization condition that $\beta^0$
lies in the column space of $M_-^T$ (e.g., $\beta^0=0$) such that
$\beta^{k+1}$ also lies in the column space of $M_-^T$; the reason
will be given in Section \ref{sec: rate_analysis}.

Indeed, (\ref{eq:beta-update}) also leads to a simple
decentralized algorithm that involves only an $x$-update and a new
multiplier update. To see this, substituting
$\frac{1}{2}M_+^Tx^k-z^k=0$ into the first two equations of
(\ref{eq:beta-update}) we have
\begin{equation} \label{eq:beta-update-temp}
\begin{array}{cr}
& \nabla f(x^{k+1}) + M_- \beta^{k+1} - \frac{c}{2}M_+M_+^Tx^k + \frac{c}{2}M_+M_+^Tx^{k+1}=0, \\
& \beta^{k+1}-\beta^k-\frac{c}{2}M_-^Tx^{k+1}=0, \\
\end{array}
\end{equation}
which is irrelevant with $z$. Note that in the first equation of
(\ref{eq:beta-update-temp}) the $x$-update relies on $M_-
\beta^{k+1}$ other than $\beta^k$. Therefore, multiplying the
second equation with $M_-$ we have $M_- \beta^{k+1}-
M_-\beta^k-\frac{c}{2}M_-M_-^Tx^{k+1}=0$. Substituting it to the
first equation of (\ref{eq:beta-update-temp}) we obtain the
$x$-update where $x^{k+1}$ is decided by $x^k$ and $M_-\beta^k$,
i.e., $\nabla f(x^{k+1}) + M_- \beta^{k} + (\frac{c}{2}M_+M_+^T +
\frac{c}{2}M_-M_-^T)x^{k+1} - \frac{c}{2}M_+M_+^Tx^k=0$. Letting
$W \in \mathbb{R}^{LN \times LN}$ be a block diagonal matrix with
its $(i,i)$th block being the degree of agent $i$ multiplying
$I_N$ and other blocks being $0_N$, $L_+=\frac{1}{2}M_+M_+^T$,
$L_-=\frac{1}{2}M_-M_-^T$, we know $W=\frac{1}{2}(L_++L_-)$.
Defining a new multiplier $\alpha=M_-\beta \in \mathbb{R}^{LN}$,
we obtain a simplified decentralized algorithm
\begin{equation} \label{eq:alpha-update}
\begin{array}{cr}
x\text{-update:} & \nabla f(x^{k+1}) + \alpha^k + 2cWx^{k+1} - cL_+x^k=0, \\
\alpha\text{-update:} & \alpha^{k+1}-\alpha^k-cL_-x^{k+1}=0. \\
\end{array}
\end{equation}

The introduced matrices $M_+$, $M_-$, $L_+$, $L_-$, and $W$ are
related to the underlying network topology. With regard to the
undirected graph $\mathcal{G}_\mathrm{u}$, $M_+$ and $M_-$ are the extended
unoriented and oriented incidence matrices, respectively; $L_+$
and $L_-$ are the extended signless and signed Laplacian matrices,
respectively; and $W$ is the extended degree matrix. By
``extended'', we mean replacing every $1$ by $I_N$, $-1$ by
$-I_N$, and $0$ by $0_N$ in the original definitions of these
matrices \cite{Chung1997,Fiedler1973,Cvetkovic2007,Chen2010}.

The updates in (\ref{eq:alpha-update}) are distributed to agents.
Note that $x=[x_1;\cdots;x_L]$ where $x_i$ is the local solution
of agent $i$ and $\alpha=[\alpha_1;\cdots;\alpha_L]$ where
$\alpha_i \in \mathbb{R}^N$ is the local Lagrange multiplier of
agent $i$. Recalling the definitions of $W$, $L_+$ and $L_-$,
(\ref{eq:alpha-update}) translates the update of agent $i$ by
\begin{equation} \label{eq:distributed}
\begin{array}{cr}
&\nabla f_i(x_i^{k+1})+\alpha_i^{k}+2c|\mathcal{N}_i|x_i^{k+1}-c\left(|\mathcal{N}_i|x_i^{k}+\sum\limits_{j\in\mathcal{N}_i} x_j^{k}\right)=0, \\
&\alpha_i^{k+1}=\alpha_i^{k}+c\left(|\mathcal{N}_i|x_i^{k+1}-\sum\limits_{j\in\mathcal{N}_i} x_j^{k+1}\right), \\
\end{array}
\end{equation}
where $\mathcal{N}_i$ denotes the set of neighbors of agent $i$.
The algorithm is fully decentralized since the updates of $x_i$
and $\alpha_i$ only rely on local and neighboring information. The
decentralized consensus optimization algorithm based on the ADMM
is outlined in Table \ref{tab: DCO_ADMM}.

\begin{table}
  \begin{center}
  \caption{\textbf{Algorithm 1:} Decentralized Consensus Optimization based on the ADMM}
    \vspace{-1.5em}
    \begin{tabular}{l}
    \hline
    Input functions $f_i$; initialize variables $x_i^0=0$, $\alpha_i^0=0$; set algorithm parameter $c>0$; \\
    For $k=0,1,\cdots$, every agent $i$ do \\
    \ \ \ \ Update $x_i^{k+1}$ by solving $\nabla f_i(x_i^{k+1})+\alpha_i^{k}+2c|\mathcal{N}_i|x_i^{k+1}-c\left(|\mathcal{N}_i|x_i^{k}+\sum\limits_{j\in\mathcal{N}_i} x_j^{k}\right)=0$;\\
    \ \ \ \ Update $\alpha_i^{k+1}=\alpha_i^{k}+c\left(|\mathcal{N}_i|x_i^{k+1}-\sum\limits_{j\in\mathcal{N}_i} x_j^{k+1}\right)$;\\
    End for\\
    \hline
    \end{tabular}
    \label{tab: DCO_ADMM}
  \end{center}
\end{table}

\section{Convergence Rate Analysis}\label{sec: rate_analysis}

This section first establishes the linear convergence rate of the
ADMM in decentralized consensus optimization with strongly convex
local objective functions (see Section \ref{sec:3a}); the detailed
proof of the main theoretical result is placed in Appendix. We then discuss how to tune the parameter and accelerate the convergence (see Section \ref{sec:3b}).

\subsection{Main Theoretical Result}\label{sec:3a}

Throughout this paper, we make the following assumption that the
local objective functions are strongly convex and have Lipschitz
continuous gradients; note that the latter implies
differentiability.

\begin{assumption}\label{ass: functions}
The local objective functions are strongly convex. For each agent
$i$ and given any $\tilde{x}_a, \tilde{x}_b \in \mathbb{R}^N$
$\langle \nabla f_i(\tilde{x}_a) - \nabla f_i(\tilde{x}_b),
\tilde{x}_a-\tilde{x}_b\rangle \geq m_{f_i} \|\tilde{x}_a -
\tilde{x}_b\|_2^2$  with $m_{f_i}> 0$. The gradients of the local
objective functions are Lipschitz continuous. For each agent $i$
and given any $\tilde{x}_a, \tilde{x}_b \in \mathbb{R}^N$,
$\|\nabla f_i(\tilde{x}_a) - \nabla f_i(\tilde{x}_b)\|_2 \leq
M_{f_i} \|\tilde{x}_a - \tilde{x}_b\|_2$ with $M_{f_i}
> 0$.
\end{assumption}

Recall the definition $f(x)=\sum_{i=1}^L f_i(x_i)$. Assumption
\ref{ass: functions} directly indicates that $f(x)$ is strongly
convex (i.e., $\langle \nabla f(x_a) - \nabla f(x_b), x_a -
x_b\rangle \geq m_f \|x_a - x_b\|_2^2$ given any $x_a, x_b \in
\mathbb{R}^{LN}$ with $m_f = \min_i m_{f_i}$) and the gradient of
$f(x)$ is Lipschitz continuous (i.e., $\|\nabla f(x_a) - \nabla
f(x_b)\|_2 \leq M_f \|x_a - x_b\|_2$ for any $x_a, x_b \in
\mathbb{R}^{LN}$ with $M_f = \max_i M_{f_i}$).

Although the convergence of Algorithm 1 to the optimal solution of
(\ref{eq:nc-mat}) can be shown based on the convergence property
of the ADMM (see e.g., \cite{Bertsekas1997}), establishing its
linear convergence is nontrivial. In \cite{Deng2013} the linear
convergence of the centralized ADMM is proved given that either
$g(z)$ is strongly convex or $B$ is full row-rank in
\eqref{eq:nc-mat}. However, the decentralized consensus
optimization problem does not satisfy these conditions. The
function $g(z)=0$ is not strongly convex, and the matrix
$B=[-I_{2EN}; -I_{2EN}]$ is row-rank deficient.

Next we will analyze the convergence rate of the ADMM iteration
(\ref{eq:beta-update}). The analysis requires an extra
initialization condition that $\beta^0$ lies in the column space
of $M_-^T$ such that $\beta^{k+1}$ also lies in the column space
of $M_-^T$, which is necessary in the analysis. Note that there is
a unique optimal multiplier $\beta^*$ lying in the column space of
$M_-^T$. To see so, consider the KKT conditions of
(\ref{eq:nc-mat})
\begin{equation} \label{eq:KKT-new}
\begin{array}{cr}
& \nabla f(x^*) + M_- \beta^*=0, \\
& M_-^T x^*=0, \\
& \frac{1}{2}M_+^T x^* - z^*=0, \\
\end{array}
\end{equation}
where ($x^*$, $z^*$) is the unique primal optimal solution and the
uniqueness follows from the strong convexity of $f(x)$ as well as
the consensus constraint $Ax+Bz=0$. Since the consensus
constraints $Ax + Bz = 0$ are feasible, there is at least one
optimal multiplier $\tilde{\beta}$ exists such that $\nabla f(x^*)
+ M_- \tilde{\beta}=0$. We show that its projection onto the
column space of $M_-^T$, denoted by $\beta^*$, is also an optimal
multiplier. According to the property of projection, $M_-
(\tilde{\beta} - \beta^*) = 0$ and hence $M_- \tilde{\beta} = M_-
\beta^*$. Therefore, the projection $\beta^*$ that lies in the
column space of $M_-^T$ also satisfies $\nabla f(x^*) + M_-
\beta^*=0$. Next we show the uniqueness of such a $\beta^*$ by
contradiction. Consider two different vectors $M_-^T v_1, M_-^T
v_2 \in \mathbb{R}^{2EN}$ that both lie in the column space of
$M_-^T$ and satisfy the equation. Therefore, we have $\nabla
f(x^*) + M_- M_-^T v_1=0$ and $\nabla f(x^*) + M_- M_-^T v_2=0$.
Subtracting them yields $M_- M_-^T (v_1 - v_2) = 0$. Since $\|M_-
M_-^T (v_1 - v_2)\|_2 \geq \tilde{\sigma}_{\min}(M_-) \|M_-^T (v_1
- v_2)\|_2$ where $\tilde{\sigma}_{\min}(M_-)$ is the smallest
nonzero singular value of $M_-$, we conclude that $\|M_-^T (v_1 -
v_2)\|_2=0$ and consequently $M_-^T v_1 = M_-^T v_2$ which
contradicts with the assumption of $M_-^T v_1$ and $M_-^T v_2$
being different. Hence, $\beta^*$ is the unique dual optimal
solution that lies in the column space of $M_-^T$.

Our main theoretical result considers the convergence of a vector that
concatenating the primal variable $z$ and the dual variable
$\beta$, which is common in the convergence rate analysis of the
ADMM \cite{Yuan2012,Deng2013,Hong2013}. Let us introduce
\begin{equation} \label{eq:ug}
 u=\left(
             \begin{array}{c}
               z \\
               \beta
             \end{array}
           \right),
G=\left(
             \begin{array}{cc}
               cI_{2EN} & 0_{2EN} \\
               0_{2EN} & \frac{1}{c}I_{2EN}  \\
             \end{array}
           \right).
\end{equation}
We will show that $u^k=[z^k;\beta^k]$ is Q-linearly convergent to
its optimal $u^*=[z^*;\beta^*]$ with respect to the $G$-norm.
Further, the Q-linear convergence of $u^k=[z^k;\beta^k]$ to
$u^*=[z^*;\beta^*]$ implies that $x^k$ is R-linearly convergent to
its optimal $x^*$.

\begin{theorem}\label{theorem1}
Consider the ADMM iteration (\ref{eq:beta-update}) that solves
(\ref{eq:nc-mat}). The primal variables $x$ and $z$ have their
unique optimal values $x^*$ and $z^*$, respectively; the dual
variable $\beta$ has its unique optimal value $\beta^*$ that lies
in the column space of $M_-^T$. Recall the definition of $u$ and
$G$ defined in (\ref{eq:ug}). If the local objective functions
satisfy Assumption \ref{ass: functions} and the dual variable
$\beta$ is initialized such that $\beta^0$ lies in the column
space of $M_-^T$, then for any $\mu > 1$, $u^k=[z^k;\beta^k]$ is
Q-linearly convergent to its optimal $u^*=[z^*;\beta^*]$ with
respect to the $G$-norm
\begin{equation}\label{eq:final}
\begin{array}{cl}
\|u^{k+1}-u^*\|_G^2\leq\frac{1}{1+\delta}\|u^k-u^*\|_G^2,
\end{array}
\end{equation}
%thus
%\begin{equation}\label{eq:final_temp}
%\begin{array}{cl}
%\|u^k-u^*\|_G^2\leq\left(\frac{1}{1+\delta}\right)^k\|u^0-u^*\|_G^2,
%\end{array}
%\end{equation}
where
\begin{equation}\label{eq:bound1}
\begin{array}{cl}
 \delta = \min \left\{
\frac{(\mu-1)\tilde{\sigma}_{\min}^2(M_-)}{\mu
\sigma_{\max}^2(M_+)},
\frac{m_f}{\frac{c}{4}\sigma_{\max}^2(M_+)+\frac{\mu}{c}M_f^2\tilde{\sigma}_{\min}^{-2}(M_-)}\right\}>0.
\end{array}
\end{equation}
Further, $x^k$ is R-linearly convergent to $x^*$ following from
\begin{equation}\label{eq:step1-new}
\begin{array}{cl}
\|x^{k+1}-x^*\|_2^2\leq\frac{1}{m_f}\|u^k-u^*\|_G^2. \\
\end{array}
\end{equation}
\end{theorem}

\begin{proof}
See Appendix.
\end{proof}

In Theorem \ref{theorem1}, (\ref{eq:bound1}) shows that
$\|u^{k+1}-u^*\|_G^2$ is no greater than
$\frac{1}{1+\delta}\|u^k-u^*\|_G^2$ and hence $u^k$ converges to
$u^*$ Q-linearly at a rate $$\rho\leq\sqrt{\frac{1}{1+\delta}}.$$
A larger $\delta$ guarantees faster convergence. On the other
hand, $\frac{1}{1+\delta}$ is a theoretical upper bound
of the convergence rate, probably not tight. The Q-linear
convergence of $u^k$ to $u^*$ translates to the R-linear
convergence of $x$ to $x^*$ as shown in (\ref{eq:step1-new}).

\subsection{Accelerating the Convergence}\label{sec:3b}

From (\ref{eq:bound1}) we can find that the theoretical
convergence rate (more precisely, its upper bound) is given in terms of the network topology, the properties of local objective
functions, and the algorithm parameter. The value of $\delta$ is
related with the free parameter $\mu > 1$, $\sigma_{\max}(M_+)$, $\tilde{\sigma}_{\min}(M_-)$, the strongly convexity constant $m_f$ of $f$, the
Lipschitz constant $M_f$ of $\nabla f$, and the algorithm
parameter $c$.

Now we consider tuning the free parameter $\mu$ and the algorithm
parameter $c$ to maximize $\delta$ and thus accelerate the
convergence (i.e., through minimizing $\frac{1}{1+\delta}$ that is
indeed an upper bound). From the analysis we will see more clearly
how the convergence rate is influenced by the network topology and
the local objective functions. For convenience, we define the
condition number of $f$ as
$$\kappa_f = \frac{M_f}{m_f}.$$
Recall that $m_f=\min_i m_{f_i}$ and $M_f=\max_i M_{f_i}$. Therefore,
$\kappa_f$ is an upper bound of the condition numbers of the local
objective functions. We also define the condition number of the
underlying graph $\mathcal{G}_\mathrm{d}$ or $\mathcal{G}_\mathrm{u}$ as
$$\kappa_\mathrm{G} =\frac{\sigma_{\max}(M_+)}{\tilde{\sigma}_{\min}(M_-)}=\sqrt{\frac{\sigma_{\max}(L_+)}{\tilde{\sigma}_{\min}(L_-)}}.$$
With regard to the underlying graph, the minimum nonzero singular
value of the extended signed Laplacian matrix $L_-$, denoted as
$\tilde{\sigma}_{\min}(L_-)$, is known as its algebraic
connectivity \cite{Chung1997,Fiedler1973}. The maximum singular
value of the extended signless Laplacian matrix $L_+$, denoted as
$\sigma_{\max}(L_+)$, has also drawn research interests recently
\cite{Cvetkovic2007,Chen2010}. Both $\sigma_{\max}(L_+)$ and
$\tilde{\sigma}_{\min}(L_-)$ are measures of network connectedness
but the former is weaker. Roughly speaking, larger
$\sigma_{\max}(L_+)$ and $\tilde{\sigma}_{\min}(L_-)$ mean
stronger connectedness, and a larger $\kappa_\mathrm{G}$ means
weaker connectedness.

%(if $f(x)$ is twice continuously differentiable, then $\kappa_f$
%is the condition number of its Hessian)

Keeping the definitions of $\kappa_f$ and $\kappa_\mathrm{G}$ in mind, the
following theorem shows how to choose the free  parameter $\mu$
and the algorithm parameter $c$ to maximize $\delta$ and
accelerate the convergence.

\begin{theorem}\label{theorem2}
If the algorithm parameter $c$ in (\ref{eq:bound1}) is chosen as
\begin{equation}\label{eq:c}
\begin{array}{cl}
c=c_\mathrm{t}=\frac{2\mu^{\frac{1}{2}}M_f}{\sigma_{\max}(M_+)\tilde{\sigma}_{\min}(M_-)}
\end{array}
\end{equation}
where
\begin{equation}\label{eq:mu}
\begin{array}{cl}
\mu=\left(1+\frac{\kappa_\mathrm{G}^2}{2\kappa_f^2}-\frac{\kappa_\mathrm{G}}{2\kappa_f}\sqrt{\frac{\kappa_\mathrm{G}^2}{\kappa_f^2}+4}\right)^{-1}>1,
\end{array}
\end{equation}
then
\begin{equation}\label{eq:bound3}
\begin{array}{cl}
\delta = \delta_\mathrm{t} =
\frac{1}{2\kappa_f}\sqrt{\frac{1}{\kappa_f^2}+\frac{4}{\kappa_\mathrm{G}^2}}-\frac{1}{2\kappa_f^2}
\end{array}
\end{equation}
maximizes the value of $\delta$ in (\ref{eq:bound1}) and ensures
that (\ref{eq:step1-new}) holds.
\end{theorem}

\begin{proof}
Observing the two values inside the minimization operator in
(\ref{eq:bound1}), we find that only the second term is relevant
with $c$. It is easy to check that the value of $c$ in
(\ref{eq:c}), no matter how $\mu$ is chosen, maximizes $\delta$ as
\begin{equation}\label{eq:bound2}
\begin{array}{cl}
\delta = \min \left\{
\frac{(\mu-1)\tilde{\sigma}_{\min}^2(M_-)}{\mu
\sigma_{\max}^2(M_+)}, \frac{m_f
\tilde{\sigma}_{\min}(M_-)}{\mu^{\frac{1}{2}} M_f
\sigma_{\max}(M_+)} \right\}.
\end{array}
\end{equation}
Inside the minimization operator in (\ref{eq:bound2}), the first
and second terms are monotonically increasing and decreasing with
regard to $\mu>1$, respectively. To maximize $\delta$, we choose a
value of $\mu$ such that the two terms are equal. Simple
calculations show that the value of $\mu$ in (\ref{eq:mu}), which
is larger than 1, satisfies this condition. The resulting maximum
value of $\delta$ is the one in (\ref{eq:bound3}).
\end{proof}

The value of $\delta$ in (\ref{eq:bound3}) is monotonically
decreasing with regard to $\kappa_f \geq 1$ and $\kappa_\mathrm{G} > 0$.
This conclusion suggests that a smaller condition number
$\kappa_f$ of $f(x)$ and a smaller condition number $\kappa_\mathrm{G}$ of
the graph lead to faster convergence. On the other hand, if these
condition numbers keep increasing, the convergence can go
arbitrarily slow. In fact, the limit of $\delta$ in
(\ref{eq:bound3}) is 0 as $\kappa_f \rightarrow \infty$ or
$\kappa_\mathrm{G} \rightarrow \infty$. Given $\delta_\mathrm{t}$, the upper bound of $\delta$, we define the upper bound of
the convergence rate as $$\rho_\mathrm{t} = \sqrt{\frac{1}{1+\delta_\mathrm{t}}}.$$

\section{Numerical Experiments}\label{sec: num_exp}

In this section, we provide extensive numerical experiments and supplement to
validate our theoretical analysis. We introduce experimental settings in Section \ref{sec:4a} and then study the influence of different
factors on the convergence rate in Sections \ref{sec:4b}
through \ref{sec:4e}.

\subsection{Experimental Settings} \label{sec:4a}

We generate a network consisting of $L$ agents and possessing at
most $\frac{L(L-1)}{2}$ edges. If the network is randomly
generated, we define $p$, the connectivity ratio of the network,
as its actual number of edges divided by $\frac{L(L-1)}{2}$. Such a random network is generated with $\frac{L(L-1)}{2}p$ edges that are uniformly randomly chosen, while ensuring the network
connected.

We apply the ADMM to a decentralized consensus least
squares problem
\begin{equation}\label{eq:LS1}
\begin{array}{cl}
\min\limits_{\tilde{x}}
\sum\limits_{i=1}^{L}\frac{1}{2}\|v_i-U_i\tilde{x}\|_2^2.
\end{array}
\end{equation}
Here $\tilde{x}\in\mathbb{R}^3$ is the unknown signal to estimate
and its true values follow the normal distribution $\mathcal{N}(0,I)$,
$U_i\in\mathbb{R}^{3\times3}$ is the linear measurement matrix of
agent $i$ whose elements follow $\mathcal{N}(0,1)$ by default, and $v_i\in\mathbb{R}^3$ is the
measurement vector of agent $i$ whose elements are polluted by
random noise following $\mathcal{N}(0,0.1)$.
In Section \ref{sec:4d} the elements of the matrices $U_i$ need to
be further manipulated to produce different condition numbers
$\kappa_f$ of the objective functions. We reformulate
(\ref{eq:LS1}) into the form of (\ref{eq:nc}) as
\begin{equation}\label{eq:LS2}
\begin{array}{cl}
\min\limits_{\{x_i\},\{z_{ij}\}} &\sum\limits_{i=1}^{L}\frac{1}{2}\|v_i-U_ix_i\|_2^2, \\
\mbox{s.t.} &x_i=z_{ij},~x_j=z_{ij},~ \forall (i,j)\in\mathcal{A}.
\end{array}
\end{equation}
The solution to (\ref{eq:LS1}) is denoted by $x^*$ in which the part of agent $i$ is denoted by
$x_i^*$. The algorithm is stopped once
$\|x^k-x^*\|_2$ reaches $10^{-15}$ or the number of iterations $k$
reaches $4000$, whichever is earlier.

In the numerical experiments, we choose to record the primal error
$\|x^k-x^*\|_2$ instead of $\|u^k-u^*\|_G$ as the latter incurs
significant extra computation when the number of agents $L$ is
large. But note that $\|x^k-x^*\|_2$ is not necessarily monotonic
in $k$. Let the transient convergence rate be
$\rho_k=\frac{\|x^k-x^*\|_2}{\|x^{k-1}-x^*\|_2}$. As $\rho_k$
fluctuates, we report the \emph{running geometric-average} rate of
convergence $\bar{\rho}_k$ given by
\begin{equation}\label{eq:rho_avg}
\begin{array}{cl}
\bar{\rho}_k&:=\left(\prod\limits_{s=1}^{k}{\rho_s}\right)^{1/k}\\
            &=\left(\frac{\|x^k-x^*\|_2}{\|x^0-x^*\|_2}\right)^{1/k}\\
            &\leq\sqrt{\frac{1}{1+\delta}}\left(\sqrt{\frac{1+\delta}{m_f}}\frac{\|u^0-u^*\|_G}{\|x^0-x^*\|_2}\right)^{1/k},
\end{array}
\end{equation}
which follows from (\ref{eq:final}) and (\ref{eq:step1-new}).
While $u^0$, $u^*$, $x^0$  $x^*$, and $m_f$ influence
$\bar{\rho}_k$, observing
$$\bar{\rho}:=\lim\limits_{k\rightarrow+\infty}\bar{\rho}_k\leq\sqrt\frac{1}{1+\delta},$$
we see that their influence diminishes and the steady state
$\bar{\rho}$ is upper bounded by $\sqrt\frac{1}{1+\delta}$ as $\rho$ is. Throughout the numerical experiments, we report
$\bar{\rho}_k$ and $\bar{\rho}$.

In the following subsections, we demonstrate how different factors
influence the convergence rate. We firstly show the evidence of
linear convergence, and along the way, the influence of the
connectivity ratio $p$ on the convergence rate (see Section
\ref{sec:4b}). Secondly, we compare the practical convergence rate
using the best theoretical algorithm parameter $c=c_\mathrm{t}$ in
(\ref{eq:c}) and that using the
best hand-tuned parameter $c=c^*$ (see Section \ref{sec:4c}).
Thirdly, we check the effect of $\kappa_f$, the condition number
of the objective function (see Section \ref{sec:4d}). Finally, we
show how $\kappa_f$, the condition number of the network, as well
as other network parameters, influence the convergence rate (see
Section \ref{sec:4e}). The numerical experiments are summarized in Table II.

\begin{table}\label{tab: sumup}
\centering\caption{Summary of the Numerical Experiments}
\begin{tabular}{c|c|c}
\hline\hline
Section&Factor&Conclusion\\
\hline\hline
\ref{sec:4b}&$p$, connectivity ratio&Larger $p$ leads to faster convergence\\
\hline
\ref{sec:4c}&$c$, algorithm parameter &$c \simeq 0.5c_\mathrm{t}$ works well \\
\hline
\ref{sec:4d}&$\kappa_f$, condition number of objective function &Larger $\kappa_f$ leads to slower convergence\\
\hline
\ref{sec:4e}&$\kappa_\mathrm{G}$, condition number of network &Larger $\kappa_\mathrm{G}$ leads to slower convergence\\
\hline
\ref{sec:4e}&$D$, network diameter &Larger $D$ leads to slower convergence\\
\hline
\ref{sec:4e}&$d_s$, geometric average degree &Larger $d_s$ leads to faster convergence\\
\hline
\ref{sec:4e}&$L_d$, imbalance of bipartite graphs &Larger $L_d$ leads to faster convergence\\
\hline
\end{tabular}
\end{table}

\begin{figure}
\begin{center}
\includegraphics[width=8cm]{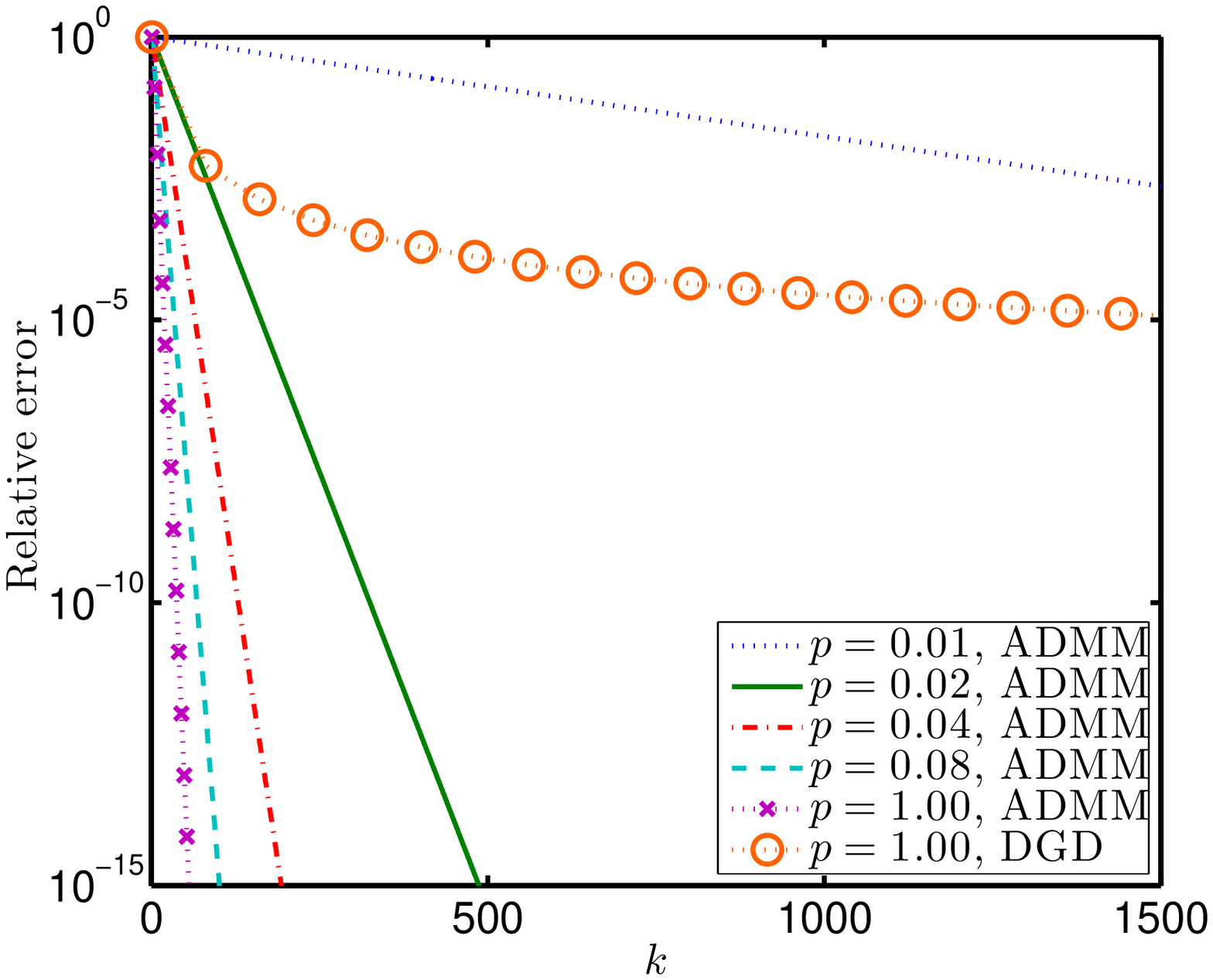}
\caption{Relative error $\frac{\|x^k-x^*\|_2}{\|x^*\|_2}$ versus
iteration $k$.} \label{eps: ADMM_Linear_Convergence-1}
\end{center}
\end{figure}

\begin{figure}
\begin{center}
\includegraphics[width=8cm]{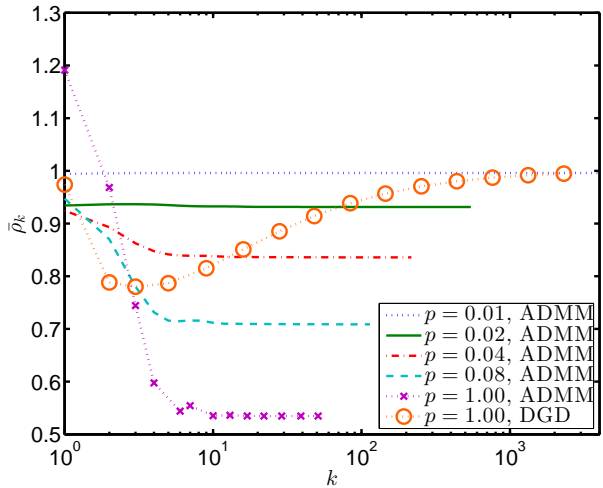}
\caption{Running geometric-average rate of convergence
$\bar{\rho}_k$ versus iteration $k$.} \label{eps:
ADMM_Linear_Convergence-2}
\end{center}
\end{figure}

\subsection{Linear Convergence}\label{sec:4b}

To illustrate linear convergence of the ADMM for decentralized
consensus optimization, we generate random networks consisting of
$L=200$ agents. The connectivity ratio of the networks, $p$, is set to
different values. The ADMM parameter is set as $c=c_\mathrm{t}$ (\ref{eq:c}).

Fig. \ref{eps: ADMM_Linear_Convergence-1} depicts how the relative
error, $\frac{\|x^k-x^*\|_2}{\|x^*\|_2}$, varies in $k$. Obviously
the convergence rates are linear for all $p$; a higher
connectivity ratio leads to faster convergence. Fig. \ref{eps:
ADMM_Linear_Convergence-2} plots $\bar{\rho}_k$, which stabilizes
within $10$ iterations. From Fig. \ref{eps:
ADMM_Linear_Convergence-1} and Fig. \ref{eps:
ADMM_Linear_Convergence-2}, one can observe that for such randomly
generated networks, varying the connectivity ratio $p$ within the
range $[0.08, 1]$ does not significantly change the convergence
rate. The reason is that when $p$ is larger than a certain
threshold, its value makes little influence on $\kappa_\mathrm{G}$
(see Table III in Section \ref{sec:4c}). We will discuss more
about the influence of $\kappa_\mathrm{G}$ in Section
\ref{sec:4d}.

As a comparison, we also demonstrate the convergence of the
distributed gradient descent (DGD) method in Fig. \ref{eps:
ADMM_Linear_Convergence-1} and Fig. \ref{eps:
ADMM_Linear_Convergence-2}. Using a diminishing stepsize
$1/k^{1/3}$ \cite{Jakovetic2012}, the DGD shows sublinear
convergence that is slow even for a complete graph (i.e., $p=1$).

\subsection{Algorithm Parameter}\label{sec:4c}

Here we discuss the influence of the ADMM parameter $c$ on
the convergence rate. The best theoretical value $c=c_\mathrm{t}$ in
(\ref{eq:c}), though optimizing the upper bound of the convergence
rate, does not give best practical performance. We vary $c$,
and plot the steady-state running geometric-average rates of
convergence $\bar{\rho}$ in Fig. \ref{eps: traversing_c}. For each
curve that corresponds to a unique $p$, we mark the
best theoretical value $c_\mathrm{t}$ and the best practical value
$c^*$. Consistently, $c_\mathrm{t}$ are larger than $c^*$.

\begin{figure}
\begin{center}
\includegraphics[width=8cm]{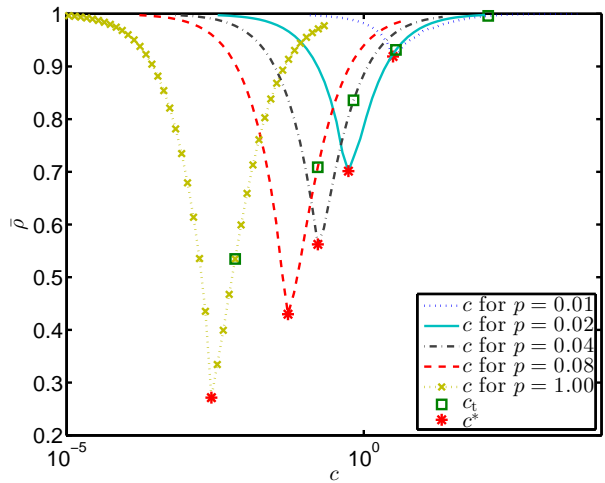}
\caption{Steady-state running geometric-average rate of
convergence $\bar{\rho}$ versus algorithm parameter
$c$.}\label{eps: traversing_c}
\end{center}
\end{figure}

\begin{figure}
\begin{center}
\includegraphics[width=8cm]{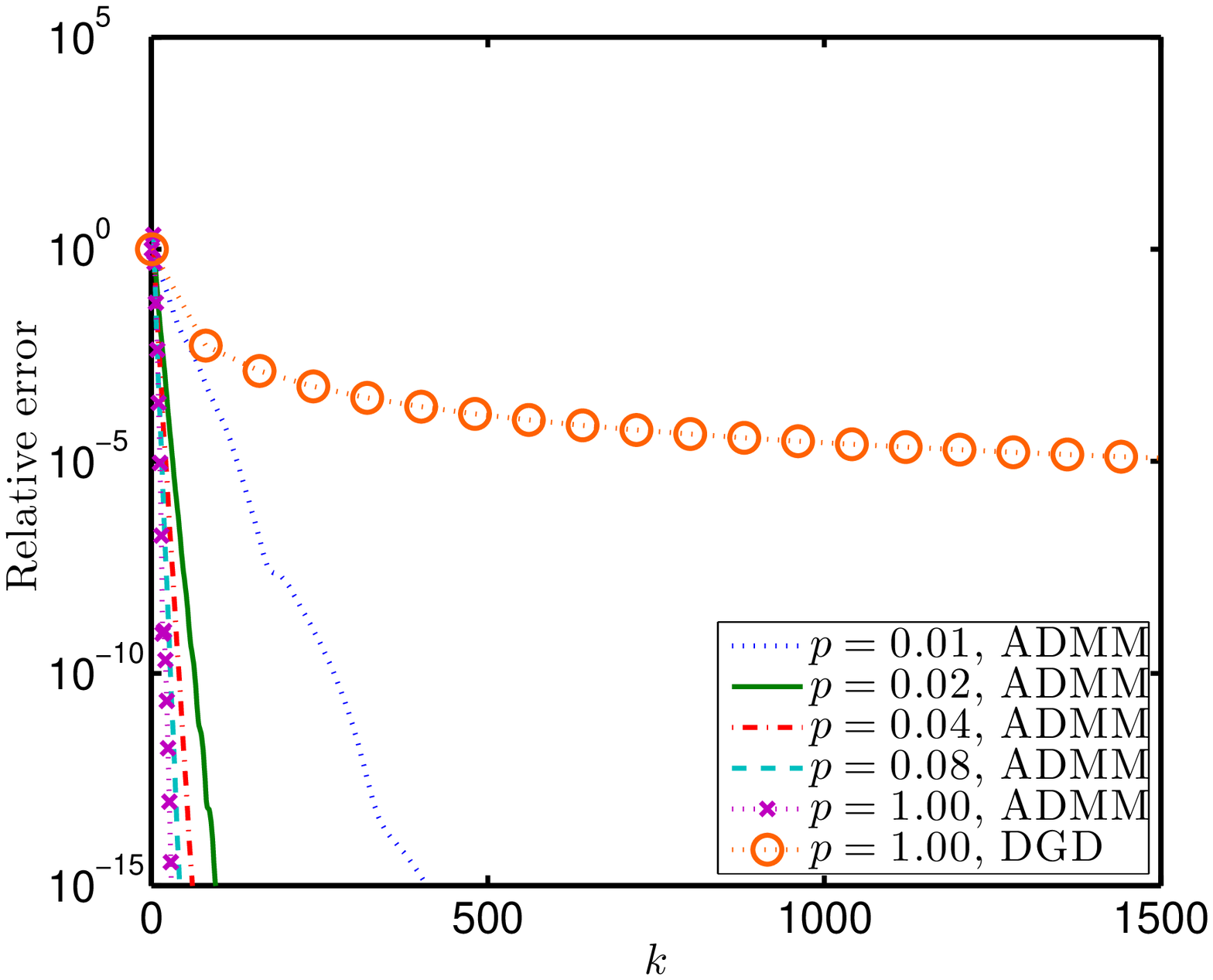}
\caption{Relative error $\frac{\|x^k-x^*\|_2}{\|x^*\|_2}$
versus iteration $k$.} \label{eps: best_c-1}
\end{center}
\end{figure}

\begin{figure}
\begin{center}
\includegraphics[width=8cm]{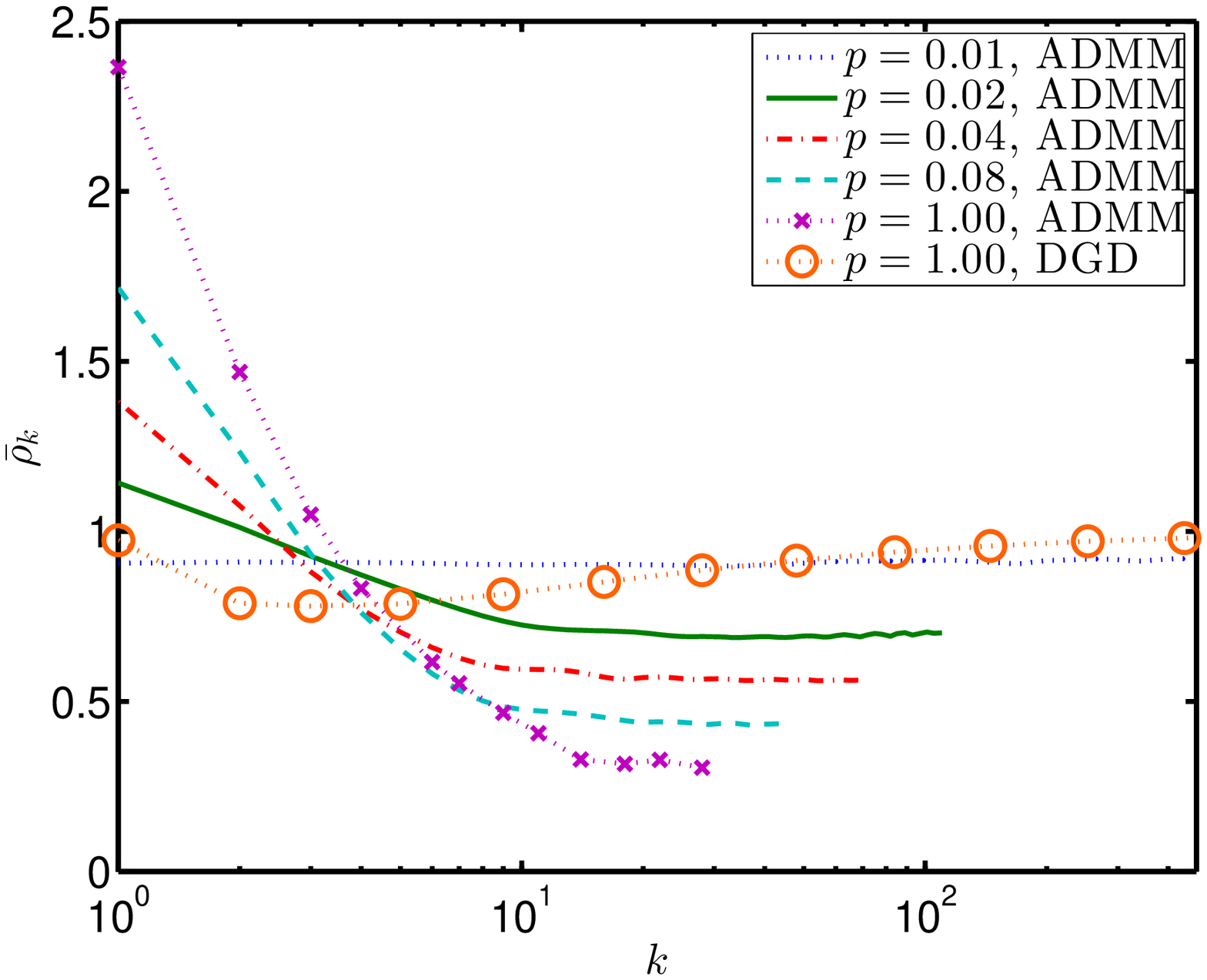}
\caption{Running geometric-average rate of convergence
$\bar{\rho}_k$ versus iteration $k$.} \label{eps:
best_c-2}
\end{center}
\end{figure}

Now we set $c=c^*$, the hand-tuned optimal value, and plot $\frac{\|x^k-x^*\|_2}{\|x^*\|_2}$
in Fig. \ref{eps: best_c-1} as per Fig. \ref{eps:
ADMM_Linear_Convergence-1} and $\bar{\rho}_k$ in Fig. \ref{eps: best_c-2} as per
Fig. \ref{eps: ADMM_Linear_Convergence-2}. Comparing to those
using $c=c_\mathrm{t}$, the best theoretical value, in Fig. \ref{eps:
ADMM_Linear_Convergence-1} and Fig. \ref{eps:
ADMM_Linear_Convergence-2}, the convergence improves
significantly. The numerical quantities of Figs. \ref{eps:
ADMM_Linear_Convergence-1}, \ref{eps:
ADMM_Linear_Convergence-2}, \ref{eps: best_c-1}, and \ref{eps: best_c-2} are given in Table III.

It appears that $c_\mathrm{t}$ is a stable overestimate of $c^*$.
Therefore, we recommend $c=\theta c_\mathrm{t}$ for nearly optimal
convergence using some $\theta \in (0,1)$. Fig. \ref{eps:
effect_of_beita} illustrates the convergence corresponding to
different values of $\theta$. We randomly generate $4000$
connected networks with $L=200$ agents whose connectivity ratios
are uniformly distributed on $[\frac{2}{L},1]$. The random
networks are divided into $20$ groups according to their condition
numbers $\kappa_\mathrm{G}$. For each group of the random networks, the values of $\bar{\rho}$ are plotted with error bars, and compared with the
theoretical upper bound $\rho_\mathrm{t}$. For this dataset,
$\theta \simeq 0.5$ appear to be a good overall choice. A smaller
$\theta$ imposes a risk of slower convergence when $\kappa_\mathrm{G}$ is
small.

\begin{table}
\centering\caption{Settings and convergence rates corresponding to
Fig. \ref{eps: ADMM_Linear_Convergence-1}, Fig. \ref{eps:
ADMM_Linear_Convergence-2}, Fig. \ref{eps: best_c-1}, and Fig.
\ref{eps: best_c-2}}
\begin{tabular}{c||c||c|c||c|c||c}
\hline\hline
Connectivity ratio $p$&          & \multicolumn{2}{|c||}{Best theoretical $c_\mathrm{t}$}& \multicolumn{2}{|c||}{Best practical $c^*$}& Theoretical rate \\
\cline{3-6}
($L=200$ agents) &$\kappa_\mathrm{G}$&$c$       &$\bar{\rho}$    &$c$       &$\bar{\rho}$    &$\rho_\mathrm{t}$\\
\hline\hline
$0.01$    &$33.00$   &$123.8$    &$0.9960$          &$3.110$   &$0.9189$          &$0.9908$\\
\hline
$0.02$    &$7.032$   &$3.477$   &$0.9314$          &$0.5510$  &$0.7014$          &$0.9806$\\
\hline
$0.04$  &$3.500$   &$0.6714$   &$0.8358$          &$0.1687$   &$0.5624$          &$0.9295$\\
\hline
$0.08$   &$2.221$   &$0.1677$ &$0.7088$          &$0.05303$ &$0.4297$          &$0.8526$\\
\hline
$1.00$     &$1.411$   &$0.006837$&$0.5348$          &$0.002722$&$0.2714$          &$0.7313$\\
\hline
\end{tabular}
\end{table}

\begin{figure}
\begin{center}
\includegraphics[width=8cm]{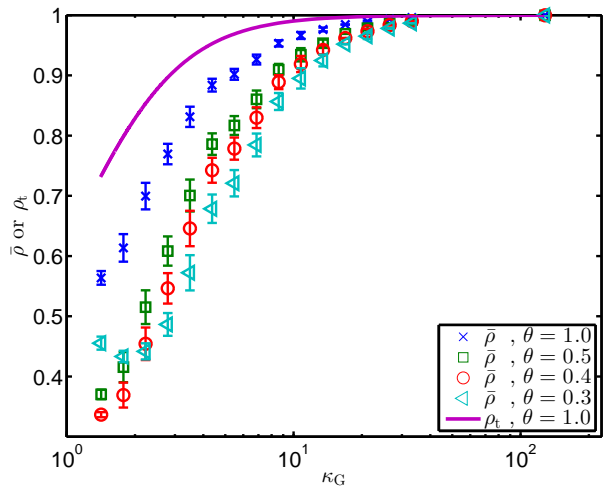}
\caption{Convergence performance obtained with $c=\theta c_\mathrm{t}$ for
varying $\theta$, where $c_\mathrm{t}$ is analytically given in
(\ref{eq:c}).}\label{eps: effect_of_beita}
\end{center}
\end{figure}

\subsection{Condition Number of the Objective Function}\label{sec:4d}

Now we study how $\kappa_f$, the condition number of the objective
function, affects the convergence rate. We generate random
networks consisting of $L=200$ agents with different connectivity
ratios $p$. We set $c=c_\mathrm{t}$. To produce different
$\kappa_f$, we first generate a linear measurement matrix $U_i$
with its elements following $\mathcal{N}(0,1)$. Second, we apply singular value decompositions
to $U_i$, scale the singular values to the range
$[\sqrt{\frac{1}{\kappa_f}},1]$, and rebuild $U_i$.

Fig. \ref{eps: kappa_f_rate} shows that the theoretical
convergence rates $\rho_\mathrm{t}$ are monotonically increasing
as $\kappa_f$ increases, which is consistent with Theorem
\ref{theorem2}. When the connectivity ratios $p$ are small, the
trend of $\bar{\rho}$ disobeys the theoretical analysis. It is
because that our upper bound of the convergence rate, becomes
loose when the network connectedness is poor. When the network is
well-connected (say $p=1$), we can observe a positive correlation
between $\bar{\rho}$ and $\kappa_f$, which coincides with the
theoretical analysis.

\begin{figure}[H]
\begin{center}
\includegraphics[width=8cm]{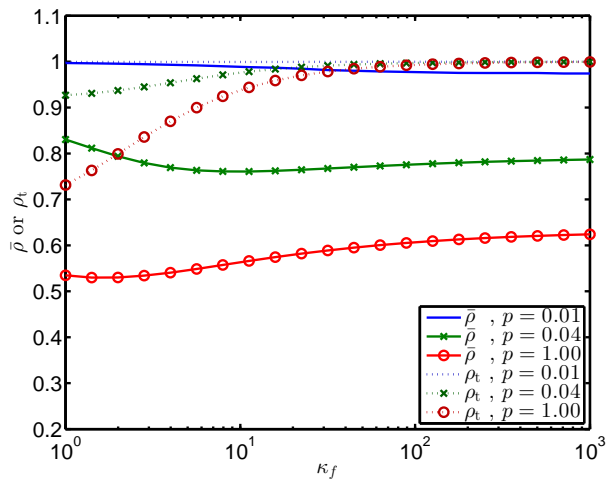}\\
\caption{Convergence performance versus the condition number
$\kappa_f$ of the objective function at different connectivity
ratios $p$.}\label{eps: kappa_f_rate}
\end{center}
\end{figure}

\subsection{Network Topology}\label{sec:4e}

Last we study how the network topology affects the convergence
rate. Besides the condition number $\kappa_\mathrm{G}$ of the network that
is relevant, we also consider other
network parameters including the network diameter, geometric average
degree, as well as imbalance of bipartite networks. In the
numerical experiments, the local objective functions are generated
as described in Section \ref{sec:4a}. The algorithm parameter is set as $c=c_\mathrm{t}$.

\subsubsection{Condition Number of the Network}

As it is difficult to precisely design $\kappa_\mathrm{G}$, the
condition number of the network, we run a large number of trials
to sample $\kappa_\mathrm{G}$. We randomly generate $4000$
connected networks with $L=50,200,500$ agents, 12000 networks in
total. Their connectivity ratios are uniformly distributed on
$[\frac{2}{L},1]$. In addition, we generate special networks with
topologies of the line, cycle, star, complete, and grid types. The
grid networks are generated in a 3D space ($2 \times 5 \times 5$,
$5 \times 5 \times 8$, and $5 \times 10 \times 10$).

Fig. \ref{eps: effect_of_L} depicts the effect of
$\kappa_\mathrm{G}$ on the convergence rate. In Fig. \ref{eps:
effect_of_L}, the dashed curve with error bars correspond to the
random networks, and the individual points correspond to the
special networks. There is only one dashed curve in the plot since
$L=50,200,500$ do not make significant differences. The networks
of the line, cycle, complete, and grid topologies generate points
in the plot that are nearly on the dashed curve, which indicates
that $\kappa_\mathrm{G}$ is a good indicator for convergence rate.
In addition, the trends of $\bar\rho$, the steady-state running
geometric-average rate of convergence, and $\rho_\mathrm{t}$, the
theoretical rate of convergence, are consistent. The points
corresponding to the three networks of the star topology are away
from the dashed side.

We observe that the convergence rate is closely related to
$\kappa_\mathrm{G}$, less to $L$. To reach a target convergence rate, one
therefore shall have a sufficiently small $\kappa_\mathrm{G}$, which in turn depends on $L$ and $p$, as well as other factors. To obtain a
sufficiently small $\kappa_\mathrm{G}$, typically, $p$ needs to be large if
$L$ is small, but not as large if $L$ is large. In other words, if
one has a network with a large number of agents (say $L=200$), a
small connectivity ratio (say $p=0.1$) will lead to a small
$\kappa_\mathrm{G}$ and thus fast convergence.

With the same $\kappa_\mathrm{G}$, the networks with the star topology have
much faster convergence than random networks. We shall discuss
this special topology at the end of this subsection.

%As we shall demonstrate below, the network diameter, the geometric
%average degree, and the imbalance in the case of bipartite
%networks.

\begin{figure}
\begin{center}
\includegraphics[width=8cm]{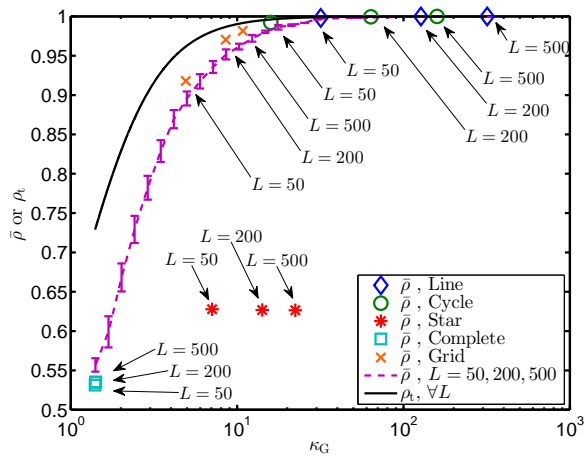}
\caption{Convergence performance versus the condition number of
the network $\kappa_\mathrm{G}$ obtained with networks of
different topologies (random, line, cycle, star, complete, and
grid) and of different sizes ($L=50,200,500$).}\label{eps:
effect_of_L}
\end{center}
\end{figure}

\subsubsection{Network Diameter}

The network diameter $D$ is defined as the longest distance
between any pair of agents in the network. In decentralized
consensus optimization, $D$ is related to how many iterations the
information from one agent will reach all the other agents.

To discuss the effect of the network diameter on the convergence
rate, we randomly generated $4000$ connected networks with $L=200$
agents and connectivity ratios uniformly distributed on
$[\frac{2}{L},1]$. We also generate the networks of the line,
cycle, star, complete, and grid topologies. Most randomly
generated networks possess small diameters. In this experiment,
the numbers of those with $D=2$, $3\leq D\leq4$ and $5\leq
D\leq198$ are $3141$, $717$ and $142$, respectively. From Fig.
\ref{eps: diameter}, we conclude that in general a larger diameter
tends to cause a worse condition number of the network and thus
slower convergence, though this relationship is interfered by
network properties.

\begin{figure}
\begin{center}
\includegraphics[width=8cm]{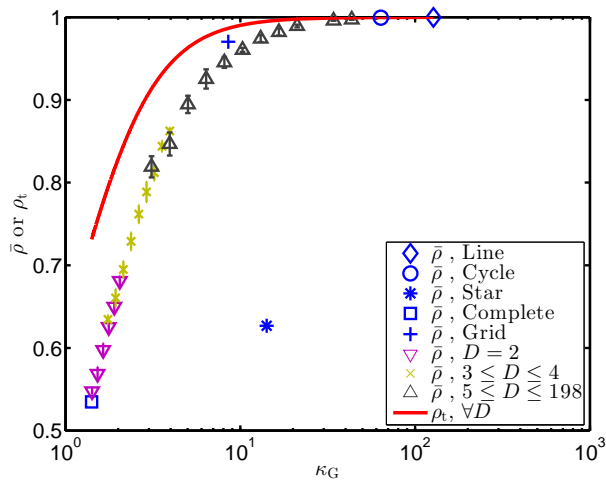}
\caption{Convergence performance versus the condition number
$\kappa_\mathrm{G}$ of the network and the network diameter $D$
obtained with networks of different topologies (random, line,
cycle, star, complete, and grid) and of size $L=200$.} \label{eps:
diameter}
\end{center}
\end{figure}

\subsubsection{Geometric Average Degree}

Define $d_{\min}$ and $d_{\max}$ as the largest and smallest
degrees of the agents in the network, respectively. The geometric
average degree $d_s=\sqrt{d_{\min}d_{\max}}$ reflects the agents'
number of neighbors in a geometric average sense. Its value
reaches maximum at $L-1$ if the topology is complete; and reaches
minimum $\sqrt{2}$ when the topology is a line.

Again, we randomly generated $4000$ connected networks with
$L=200$ agents and connectivity ratios uniformly distributed on
$[\frac{2}{L},1]$. We also generate the networks of the line,
cycle, star, complete, and grid topologies. Out of the randomly
generated networks, $417$ have $2\leq d_s\leq 20$, $1576$ have
$21\leq d_s\leq 100$, and $1956$ have $101\leq d_s\leq198$. From
Fig. \ref{eps: d_s}, we observe that a larger $d_s$ generally
implies better connectedness and thus a smaller condition number
of the network as well as faster convergence. This conclusion is
similar to the one on the network diameter $D$ (see Fig. \ref{eps:
diameter}).

\begin{figure}
\begin{center}
\includegraphics[width=8cm]{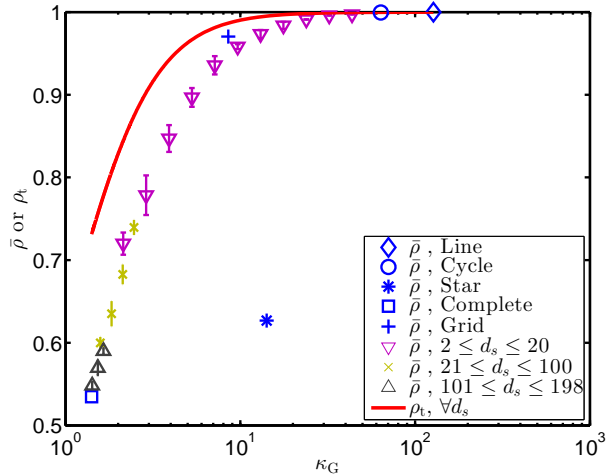}
\caption{Convergence performance versus the condition number
$\kappa_\mathrm{G}$ of the network and the geometric average
degree $d_s$ obtained with networks of different topologies
(random, line, cycle, star, complete, and grid) and of size
$L=200$.} \label{eps: d_s}
\end{center}
\end{figure}

\begin{figure}
\begin{center}
\includegraphics[width=8cm]{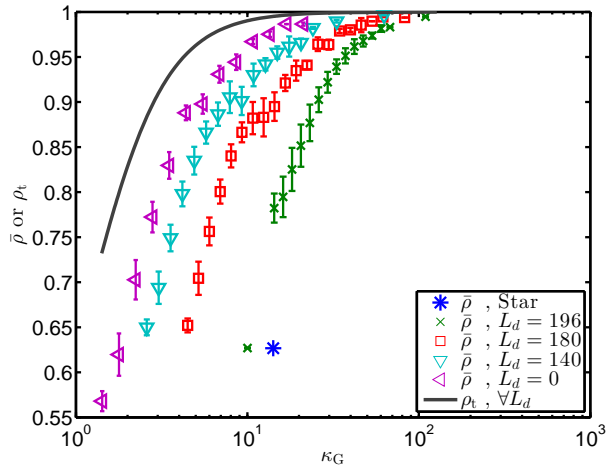}
\caption{Convergence performance versus the condition number
$\kappa_\mathrm{G}$ of the network and the imbalance of bipartite
networks $L_d$ obtained with networks of random and star
topologies and of size $L=200$.} \label{eps: study_in_bipartite}
\end{center}
\end{figure}

\subsubsection{Imbalance of Bipartite Networks}

Let $\mathcal{B}(\mathcal{L}_A,\mathcal{L}_B)$ denote the class of
bipartite networks with $|\mathcal{L}_A|$ agents in one group and
$|\mathcal{L}_B|$ agents in another group. Agents within either
group cannot directly communicate with each other. For a bipartite network
consisting of $L=|\mathcal{L}_A|+|\mathcal{L}_B|$ agents, its
imbalance is defined as $L_d = |\mathcal{L}_A|-|\mathcal{L}_B|$, which can vary between $0$ and $L-2$.

%\revise{A large value of imbalance means that a small group of
%agents are leading a large group of agents, while a small value of
%imbalance implies that two even groups cooperate using only
%intra-group communication.}

We randomly generate $1000$ bipartite graphs of size $L=200$,
whose connectivity ratios $p$ are uniformly distributed on
$[\frac{2}{L},\frac{(L+L_d)(L-L_d)}{2L(L-1)}]$, for each of the
cases $L_d=196,180,140,0$. The star topology corresponds to a
special bipartite network with $L_d=L-2=198$. From Fig. \ref{eps:
study_in_bipartite}, we find that for the same $\kappa_\mathrm{G}$, the
networks with larger $L_d$ have faster convergence. An extreme
example is the network of the star topology. This observation
suggests us to assign few ``hot spots'' to relay information for
fast convergence, if $\kappa_\mathrm{G}$ is fixed in advance. However, this
approach may cause robustness or scalability issues because the
relaying agents are subject to extensive communication burden.
Hence there is a tradeoff between fast convergence and robustness
or scalability in network design.

\section{Conclusions} \label{sec: conclusion}

We apply the ADMM to a reformulation of a general decentralized
consensus optimization problem. We show that if the objective
function is strongly convex, the decentralized ADMM converges at a
globally linear rate, which can be given explicitly. It is revealed
that several factors affect the convergence rate that include the
topology-related properties of the network, the condition number
of the objective function, and the algorithm parameter. Numerical
experiments corroborate and supplement our theoretical findings. Our analysis sheds
light on how to construct a network and tune the algorithm
parameter for fast convergence.

\appendix

%\section{Proof of Theorem \ref{theorem1}}

\begin{proof} Consider the ADMM updates (\ref{eq:beta-update}) and
the KKT conditions (\ref{eq:KKT-new}). Subtracting the three
equations in (\ref{eq:KKT-new}) from the corresponding equations
in (\ref{eq:beta-update}) yields
\begin{equation} \label{eq:v2line1}
\begin{array}{cl}
\nabla f(x^{k+1})-\nabla f(x^*) = cM_+(z^k-z^{k+1}) - M_-(\beta^{k+1}-\beta^*), \\
\end{array}
\end{equation}
\begin{equation} \label{eq:v2line2}
\begin{array}{cl}
\frac{c}{2}M_-^T(x^{k+1}-x^*)=\beta^{k+1}-\beta^k, \\
\end{array}
\end{equation}
\begin{equation} \label{eq:v2line3}
\begin{array}{cl}
\frac{1}{2}M_+^T(x^{k+1}-x^*)=z^{k+1}-z^*, \\
\end{array}
\end{equation}
respectively.

To prove the Q-linear convergence of $\|u^{k+1}-u^*\|_G^2$ we use
$m_f \|x^{k+1}-x^*\|_2^2$ as an intermediate. Based on Assumption
\ref{ass: functions}, $f(x)$ is strongly convex with a constant
$m_f$ such that
\begin{equation}\label{eq:v2main-1}
\begin{array}{cl}
     & m_f \|x^{k+1}-x^*\|_2^2 \leq \langle x^{k+1}-x^*, \nabla f(x^{k+1}) - \nabla f(x^*) \rangle. \\
\end{array}
\end{equation}
Using (\ref{eq:v2line1}), we can split the right-hand side of
(\ref{eq:v2main-1}) to two terms
\begin{equation}\label{eq:v2main-2}
\begin{array}{cl}
     & \langle x^{k+1}-x^*, \nabla f(x^{k+1}) - \nabla f(x^*) \rangle \\
=    & \langle x^{k+1}-x^*, cM_+(z^k-z^{k+1}) - M_-(\beta^{k+1}-\beta^*) \rangle \\
=    & \langle x^{k+1}-x^*, cM_+(z^k-z^{k+1}) \rangle + \langle x^{k+1}-x^*, - M_-(\beta^{k+1}-\beta^*) \rangle \\
=    & c \langle M_+^T(x^{k+1}-x^*), z^k-z^{k+1} \rangle + \langle - M_-^T(x^{k+1}-x^*), \beta^{k+1}-\beta^* \rangle. \\
\end{array}
\end{equation}
Substituting (\ref{eq:v2line2}) and (\ref{eq:v2line3}) to
(\ref{eq:v2main-2}) we can eliminate the term $x^{k+1}-x^*$ and
obtain
\begin{equation}\label{eq:v2main-3}
\begin{array}{cl}
     & \langle x^{k+1}-x^*, \nabla f(x^{k+1}) - \nabla f(x^*) \rangle \\
=    & 2c \langle z^k-z^{k+1}, z^{k+1}-z^* \rangle + \frac{2}{c} \langle \beta^k-\beta^{k+1}, \beta^{k+1}-\beta^* \rangle. \\
\end{array}
\end{equation}
Recall the definition of $u$ and $G$ defined in (\ref{eq:ug}). It
is obvious that the right-hand side of (\ref{eq:v2main-3}) can be
written as a compact form $2(u^k-u^{k+1})^T G (u^{k+1}-u^*)$. Using
the equality $2(u^k-u^{k+1})^T G (u^{k+1}-u^*) = \|u^k-u^*\|_G^2 -
\|u^{k+1}-u^*\|_G^2 - \|u^k-u^{k+1}\|_G^2$, (\ref{eq:v2main-3}) is
equivalent to
\begin{equation}\label{eq:v2main-4}
\begin{array}{cl}
     & \langle x^{k+1}-x^*, \nabla f(x^{k+1}) - \nabla f(x^*) \rangle \\
=    & \|u^k-u^*\|_G^2 - \|u^{k+1}-u^*\|_G^2 - \|u^k-u^{k+1}\|_G^2, \\
\end{array}
\end{equation}
and consequently using (\ref{eq:v2main-1})
\begin{equation}\label{eq:v2main}
\begin{array}{cl}
     & m_f \|x^{k+1}-x^*\|_2^2 \\
\leq & \|u^k-u^*\|_G^2 - \|u^{k+1}-u^*\|_G^2 - \|u^k-u^{k+1}\|_G^2. \\
\end{array}
\end{equation}

Having (\ref{eq:v2main}) at hand, to prove (\ref{eq:final}) we
only need to show
\begin{equation}\label{eq:step2}
\begin{array}{cl}
\|u^k-u^{k+1}\|_G^2 + m_f\|x^{k+1}-x^*\|_2^2 \geq \delta \|u^{k+1}-u^*\|_G^2, \\
\end{array}
\end{equation}
which is equivalent to
\begin{equation}\label{eq:step3}
\begin{array}{cl}
c\|z^{k+1}-z^k\|_2^2 + \frac{1}{c} \|\beta^{k+1}-\beta^k\|_2^2 + m_f\|x^{k+1}-x^*\|_2^2 \geq \delta c \|z^{k+1}-z^*\|_2^2 + \frac{\delta}{c} \|\beta^{k+1}-\beta^*\|_2^2. \\
\end{array}
\end{equation}
The idea of proof is to show that $\delta c \|z^{k+1}-z^*\|_2^2$
and $\frac{\delta}{c} \|\beta^{k+1}-\beta^*\|_2^2$ are upper
bounded by two non-overlapping parts of the left-hand side of
(\ref{eq:step3}), respectively.

The upper bound of $\|z^{k+1}-z^*\|_2^2$ follows from
(\ref{eq:v2line3}) that shows
$\frac{1}{2}M_+^T(x^{k+1}-x^*)=z^{k+1}-z^*$. Hence we have
\begin{equation}\label{eq:line3-new}
\begin{array}{cl}
     & \|z^{k+1}-z^*\|_2^2 \\
=    & \frac{1}{4}\|M_+^T(x^{k+1}-x^*)\|_2^2 \\
\leq & \frac{1}{4}\sigma_{\max}^2(M_+)\|x^{k+1}-x^*\|_2^2, \\
\end{array}
\end{equation}
where $\sigma_{\max}(M_+)$ is the largest singular value of $M_+$.
To find the upper bound of $\|\beta^{k+1}-\beta^*\|_2^2$, we use
two inequalities $\sigma_{\max}^2(M_+)\|z^{k+1}-z^k\|_2^2 \geq
\|M_+^T(z^{k}-z^{k+1})\|_2^2$ and $M_f^2 \|x^{k+1}-x^*\|_2^2 \geq
\|\nabla f(x^{k+1})-\nabla f(x^*)\|_2^2$; the latter holds since
$f(x)$ has Lipschitz continuous gradients with a constant $M_f$.
Therefore, given the positive algorithm parameter $c$ and any $\mu
> 1$ it holds
\begin{equation}\label{eq:line1-new-new-1}
\begin{array}{cl}
     & c^2\sigma_{\max}^2(M_+)\|z^{k+1}-z^k\|_2^2+(\mu-1)M_f^2 \|x^{k+1}-x^*\|_2^2 \\
\geq & \|cM_+^T(z^{k}-z^{k+1})\|_2^2 + (\mu-1)\|\nabla f(x^{k+1})-\nabla f(x^*)\|_2^2. \\
\end{array}
\end{equation}
Recall that from (\ref{eq:v2line1}) $cM_+(z^k-z^{k+1})$ is the
summation of $\nabla f(x^{k+1})-\nabla f(x^*)$ and
$M_-(\beta^{k+1}-\beta^*)$. Hence we can apply the basic
inequality $\|a+b\|_2^2 + (\mu-1) \|a\|_2^2 \geq (1-\frac{1}{\mu})
\|b\|_2^2$, which holds for any $\mu > 0$, to
(\ref{eq:line1-new-new-1}) and obtain
\begin{equation}\label{eq:line1-new-new-2}
\begin{array}{cl}
     & c^2\sigma_{\max}^2(M_+)\|z^{k+1}-z^k\|_2^2+(\mu-1)M_f^2 \|x^{k+1}-x^*\|_2^2 \\
\geq & (1-\frac{1}{\mu}) \|M_-(\beta^{k+1}-\beta^*)\|_2^2. \\
\end{array}
\end{equation}
Since by assumption $\beta^0$ is initialized such that it lies in
the column space of $M_-^T$, we know that $\beta^{k+1}$ lies in
the column space of $M_-^T$ too; see the ADMM updates
(\ref{eq:beta-update}). Because $\beta^*$ also lies in the column
space of $M_-^T$, $\|M_-(\beta^{k+1}-\beta^*)\|_2^2 \geq
\tilde{\sigma}_{\min}^2(M_-) \|\beta^{k+1}-\beta^*\|_2^2$ where
$\tilde{\sigma}_{\min}(M_-)$ is the smallest nonzero singular
value of $M_-$. Therefore from (\ref{eq:line1-new-new-2}) we can
upper bound $\|\beta^{k+1}-\beta^*\|_2^2$ by
\begin{equation}\label{eq:line1-new-new}
\begin{array}{cl}
     & c^2\sigma_{\max}^2(M_+)\|z^{k+1}-z^k\|_2^2+(\mu-1)M_f^2 \|x^{k+1}-x^*\|_2^2 \\
\geq & (1-\frac{1}{\mu}) \tilde{\sigma}_{\min}^2(M_-) \|\beta^{k+1}-\beta^*\|_2^2. \\
\end{array}
\end{equation}

Combining (\ref{eq:line3-new}) and (\ref{eq:line1-new-new}), we
prove (\ref{eq:step3}). From (\ref{eq:line3-new}) we have
\begin{equation}\label{eq:line3-new-final}
\begin{array}{cl}
     & \frac{c}{4}\sigma_{\max}^2(M_+)\|x^{k+1}-x^*\|_2^2 \\
\geq & c \|z^{k+1}-z^*\|_2^2. \\
\end{array}
\end{equation}
From (\ref{eq:line1-new-new}) we have
\begin{equation}\label{eq:line1-new-new-final}
\begin{array}{cl}
     & \frac{c \mu \sigma_{\max}^2(M_+)}{(\mu-1) \tilde{\sigma}_{\min}^2(M_-)}\|z^{k+1}-z^k\|_2^2+\frac{\mu M_f^2}{c \tilde{\sigma}_{\min}^2(M_-)} \|x^{k+1}-x^*\|_2^2 \\
\geq & \frac{1}{c} \|\beta^{k+1}-\beta^*\|_2^2. \\
\end{array}
\end{equation}
Summing up (\ref{eq:line3-new-final}) and
(\ref{eq:line1-new-new-final}) yields
\begin{equation}\label{eq:theorem1_temp}
\begin{array}{cl}
     & \frac{c \mu \sigma_{\max}^2(M_+)}{(\mu-1) \tilde{\sigma}_{\min}^2(M_-)}\|z^{k+1}-z^k\|_2^2+ \left(\frac{\mu M_f^2}{c \tilde{\sigma}_{\min}^2(M_-)}+\frac{c}{4}\sigma_{\max}^2(M_+)\right) \|x^{k+1}-x^*\|_2^2 \\
\geq & c \|z^{k+1}-z^*\|_2^2 + \frac{1}{c} \|\beta^{k+1}-\beta^*\|_2^2. \\
\end{array}
\end{equation}
Apparently, $\delta$ in (\ref{eq:bound1}) satisfies
\begin{equation}\label{eq:step3-1}
\begin{array}{cl}
c\|z^{k+1}-z^k\|_2^2 + m_f\|x^{k+1}-x^*\|_2^2 \geq \delta c \|z^{k+1}-z^*\|_2^2 + \frac{\delta}{c} \|\beta^{k+1}-\beta^*\|_2^2, \\
\end{array}
\end{equation}
and consequently (\ref{eq:step3}), which proves (\ref{eq:final}).

To prove the R-linear convergence of $x^{k}$ to $x^*$, we observe
that (\ref{eq:v2main}) implies $m_f \|x^{k+1}-x^*\|_2^2 \leq
\|u^k-u^*\|_G^2$, which proves (\ref{eq:step1-new}).
\end{proof}

\bibliographystyle{IEEEtran}
\bibliography{ADMM}

% Generated by IEEEtran.bst, version: 1.13 (2008/09/30)
\begin{thebibliography}{10}
\providecommand{\url}[1]{#1}
\csname url@samestyle\endcsname
\providecommand{\newblock}{\relax}
\providecommand{\bibinfo}[2]{#2}
\providecommand{\BIBentrySTDinterwordspacing}{\spaceskip=0pt\relax}
\providecommand{\BIBentryALTinterwordstretchfactor}{4}
\providecommand{\BIBentryALTinterwordspacing}{\spaceskip=\fontdimen2\font plus
\BIBentryALTinterwordstretchfactor\fontdimen3\font minus
  \fontdimen4\font\relax}
\providecommand{\BIBforeignlanguage}[2]{{%
\expandafter\ifx\csname l@#1\endcsname\relax
\typeout{** WARNING: IEEEtran.bst: No hyphenation pattern has been}%
\typeout{** loaded for the language `#1'. Using the pattern for}%
\typeout{** the default language instead.}%
\else
\language=\csname l@#1\endcsname
\fi
#2}}
\providecommand{\BIBdecl}{\relax}
\BIBdecl

\bibitem{Shi2013}
W.~Shi, Q.~Ling, K.~Yuan, G.~Wu, and W.~Yin, ``{L}inearly {C}onvergent
  {D}ecentralized {C}onsensus {O}ptimization with the {A}lternating {D}irection
  {M}ethod of {M}ultipliers,'' in \emph{Proceedings of the 38th International
  Conference on Acoustics, Speech, and Signal Processing}, 2013.

\bibitem{Inalhany2002}
G.~Inalhany, D.~Stipanovic, and C.~Tomlin, ``{D}ecentralized {O}ptimization,
  with {A}pplication to {M}ultiple {A}ircraft {C}oordination,'' in
  \emph{Proceedings of the 41st IEEE Conference on Decision and Control}, 2002.

\bibitem{Ren2007}
W.~Ren, R.~Beard, and E.~Atkins, ``{I}nformation {C}onsensus in {M}ultivehicle
  {C}ooperative {C}ontrol: {C}ollective {G}roup {B}ehavior through {L}ocal
  {I}nteraction,'' \emph{IEEE Control Systems Magazine}, vol.~27, pp. 71--82,
  2007.

\bibitem{Johansson2008}
B.~Johansson, ``{O}n {D}istributed {O}ptimization in {N}etworked {S}ystems,''
  Ph.D. dissertation, KTH, Automatic Control, 2008, {Q}{C} 20100813.

\bibitem{Xiao2007}
L.~Xiao, S.~Boyd, and S.~Kim, ``{D}istributed {A}verage {C}onsensus with
  {L}east-mean-square {D}eviation,'' \emph{Journal of Parallel and Distributed
  Computing}, vol.~67, pp. 33--46, 2007.

\bibitem{Dimakis2010}
A.~Dimakis, S.~Kar, M.~R. J.~Moura, and A.~Scaglione, ``{G}ossip {A}lgorithms
  for {D}istributed {S}ignal {P}rocessing,'' \emph{Proceedings of the IEEE},
  vol.~98, pp. 1847--1864, 2010.

\bibitem{Predd2009}
J.~Predd, S.~Kulkarni, and H.~Poor, ``{A} {C}ollaborative {T}raining
  {A}lgorithm for {D}istributed {L}earning,'' \emph{IEEE Transactions on
  Information Theory}, vol.~55, pp. 1856--1871, 2009.

\bibitem{Mateos2010}
G.~Mateos, J.~Bazerque, and G.~Giannakis, ``{D}istributed {S}parse {L}inear
  {R}egression,'' \emph{IEEE Transactions on Signal Processing}, vol.~58, pp.
  5262--5276, 2010.

\bibitem{Schizas2008}
I.~Schizas, A.~Ribeiro, and G.~Giannakis, ``{C}onsensus in {A}d hoc {W}{S}{N}s
  with {N}oisy {L}inks--{P}art {I}: {D}istributed {E}stimation of
  {D}eterministic {S}ignals,'' \emph{IEEE Transactions on Signal Processing},
  vol.~56, pp. 350--364, 2008.

\bibitem{Ling2010}
Q.~Ling and Z.~Tian, ``{D}ecentralized {S}parse {S}ignal {R}ecovery for
  {C}ompressive {S}leeping {W}ireless {S}ensor {N}etworks,'' \emph{IEEE
  Transactions on Signal Processing}, vol.~58, pp. 3816--3827, 2010.

\bibitem{Bazerque2010}
J.~Bazerque and G.~Giannakis, ``{D}istributed {S}pectrum {S}ensing for
  {C}ognitive {R}adio {N}etworks by {E}xploiting {S}parsity,'' \emph{IEEE
  Transactions on Signal Processing}, vol.~58, pp. 1847--1862, 2010.

\bibitem{Bazerque2011}
J.~Bazerque, G.~Mateos, and G.~Giannakis, ``{G}roup-lasso on {S}plines for
  {S}pectrum {C}artograph,'' \emph{IEEE Transactions on Signal Processing},
  vol.~59, pp. 4648--4663, 2011.

\bibitem{Kekatos2013}
V.~Kekatos and G.~Giannakis, ``{D}istributed {R}obust {P}ower {S}ystem {S}tate
  {E}stimation,'' \emph{IEEE Transactions on Power Systems}, vol.~28, no.~2,
  pp. 1617--1626, 2013.

\bibitem{Gan2013}
L.~Gan, U.~Topcu, and S.~Low, ``{O}ptimal {D}ecentralized {P}rotocol for
  {E}lectric {V}ehicle {C}harging,'' \emph{IEEE Transactions on Power Systems},
  vol.~28, no.~2, pp. 940--951, 2013.

\bibitem{Nedic2009}
A.~Nedic and A.~Ozdaglar, ``{D}istributed {S}ubgradient {M}ethods for
  {M}ulti-agent {O}ptimization,'' \emph{IEEE Transactions on Automatic
  Control}, vol.~54, pp. 48--61, 2009.

\bibitem{Ram2010}
S.~Ram, A.~Nedic, and V.Veeravalli, ``{D}istributed {S}tochastic {S}ubgradient
  {P}rojection {A}lgorithms for {C}onvex {O}ptimization,'' \emph{Journal of
  Optimization Theory and Applications}, vol. 147, pp. 516--545, 2010.

\bibitem{Tsianos2013}
K.~Tsianos and M.~Rabbat, ``{D}istributed {S}trongly {C}onvex {O}ptimization,''
  in \emph{Proceedings of the 50th Annual Allerton Conference on Communication,
  Control and Computing}, 2012.

\bibitem{Duchi2012}
J.~Duchi, A.~Agarwal, and M.~Wainwright, ``{D}ual {A}veraging for {D}istributed
  {O}ptimization: {C}onvergence {A}nalysis and {N}etwork {S}caling,''
  \emph{IEEE Transactions on Automatic Control}, vol.~57, no.~3, pp. 592--606,
  2012.

\bibitem{Tsianos2012}
K.~Tsianos, S.~Lawlor, and M.~Rabbat, ``{P}ush-{S}um {D}istributed {D}ual
  {A}veraging for {C}onvex {O}ptimization,'' in \emph{Proceedings of the 51st
  IEEE Annual Conference on Decision and Control}, 2012, pp. 5453--5458.

\bibitem{Erseghe2011}
T.~Erseghe, D.~Zennaro, E.~Dall'Anese, and L.~Vangelista, ``{F}ast {C}onsensus
  by the {A}lternating {D}irection {M}ultipliers {M}ethod,'' \emph{IEEE
  Transactions on Signal Processing}, vol.~59, pp. 5523--5537, 2011.

\bibitem{Bertsekas1997}
D.~Bertsekas and J.~Tsitsiklis, \emph{{Parallel and Distributed Computation:
  Numerical Methods}}, 2nd~ed.\hskip 1em plus 0.5em minus 0.4em\relax Nashua:
  Athena Scientific, 1997.

\bibitem{Rabbat2006}
M.~Rabbat and R.~Nowak, ``{Q}uantized {I}ncremental {A}lgorithms for
  {D}istributed {O}ptimization,'' \emph{IEEE Journal of Selected Areas in
  Communications}, vol.~23, pp. 798--808, 2006.

\bibitem{Yuan2012}
B.~He and X.~Yuan, ``{O}n the {$O$}$(1/t)$ {C}onvergence {R}ate of the
  {D}ouglas-{R}achford {A}lternating {D}irection {M}ethod,'' \emph{SIAM Journal
  on Numerical Analysis}, vol.~50, no.~2, pp. 700--709, 2012.

\bibitem{Hong2013}
M.~Hong and Z.~Luo, ``{O}n the {L}inear {C}onvergence of the {A}lternating
  {D}irection {M}ethod of {M}ultipliers,'' \emph{arXiv preprint
  arXiv:1208.3922}, 2012.

\bibitem{Deng2013}
W.~Deng and W.~Yin, ``{O}n the {G}lobal and {L}inear {C}onvergence of the
  {G}eneralized {A}lternating {D}irection {M}ethod of {M}ultipliers,'' Rice
  University, Tech. Rep., 2012, {C}AAM Technical Report TR12-14.

\bibitem{Chung1997}
F.~Chung, \emph{{S}pectral {G}raph {T}heory}.\hskip 1em plus 0.5em minus
  0.4em\relax American Mathematical Soc., 1997, vol.~92.

\bibitem{Fiedler1973}
M.~Fiedler, ``{A}lgebra {C}onnectivity of {G}raphs,'' \emph{Czechoslovake
  Mathematical Journnal}, vol.~23, no.~98, pp. 298--305, 1973.

\bibitem{Cvetkovic2007}
D.~Cvetkovic, P.~Rowlinson, and S.~Simic, ``{S}ignless {L}aplacians of {F}inite
  {G}raphs,'' \emph{Linear Algebra and its Applications}, vol. 423, pp.
  155--171, 2007.

\bibitem{Chen2010}
Y.~Chen and L.~Wang, ``{S}harp {B}ounds for the {L}argest {E}igenvalue of the
  {S}ignless {L}aplacian of {A} {G}raph,'' \emph{Linear Algebra and its
  Applications}, vol. 433, pp. 908--913, 2010.

\bibitem{Jakovetic2012}
D.~Jakovetic, J.~Xavier, and J.~Moura, ``{C}onvergence {R}ate {A}nalysis of
  {D}istributed {G}radient {M}ethods for {S}mooth {O}ptimization,'' in
  \emph{20th Telecommunications Forum (TELFOR)}.\hskip 1em plus 0.5em minus
  0.4em\relax IEEE, 2012, pp. 867--870.

\end{thebibliography}
\end{document}